\documentclass[11pt]{article} 

\usepackage[english]{babel}

\usepackage{amssymb,amsfonts}
\usepackage{mathtools}
\usepackage{esint}
\usepackage{amsmath,amsthm}
\usepackage{graphicx}
\usepackage{fancyhdr}
\usepackage{xcolor}

\newtheorem{proposition}{Proposition}[section]
\newtheorem{lemma}[proposition]{Lemma}
\newtheorem{remark}[proposition]{Remark}
\newtheorem*{remark*}{Remark}
\newtheorem{theorem}[proposition]{Theorem}

\newcommand{\ignore}[1]{}

\newcommand{\stern}[1]{#1_*}
\newcommand{\weakto}{\rightharpoonup}
\newcommand{\weakstarto}{\stackrel{\ast}{\rightharpoonup}}
\newcommand{\Chi}{\mathcal{X}}
\newcommand{\Ha}{\mathcal{H}}
\newcommand{\calD}{\mathcal{D}}
\newcommand{\eps}{\varepsilon}

\newcommand{\one}{\boldsymbol{1}}
\newcommand{\init}{\text{in}}
\newcommand{\R}{\mathbb{R}}
\newcommand{\N}{\mathbb{N}}
\newcommand{\embeds}{\hookrightarrow}
\newcommand{\shift}{S}
\newcommand{\Xspace}{V_2(\Gamma_T)}
\newcommand{\XspaceOm}{V_2(\Omega_T)}
\newcommand{\ssubset}{\subset\mathrel{\mkern-3mu}\subset}

\usepackage{hyperref}
\hypersetup{
  pdfauthor={},
  pdftitle={},
  pdfsubject={},
  urlcolor=blue,
}

\begin{document}

	\title{A parabolic free boundary problem arising in a model of cell polarization}

\author{A. Logioti
\and
B. Niethammer
\and
M. R\"oger
\and
J. J. L. Vel\'{a}zquez
}
\maketitle

\begin{abstract}
The amplification of an external signal is a key step in direction sensing of biological cells. We consider a simple model for the response to a time-depending signal, which was previously proposed by the last three authors. 
The model consists of a bulk-surface reaction-diffusion model. We prove that in a suitable asymptotic limit the system converges to a bulk-surface parabolic obstacle type problem. For this model and a reduction to a nonlocal surface equation we show an $L^1$-contraction property and, in the case of time-constant signals, the stability of stationary states.

\medskip%
\noindent%
\\[2ex]\noindent%
\end{abstract}

\tableofcontents

\section{Introduction}
\medspace
Cell polarization in response to some external chemical stimulus contributes significantly in numerous biological processes, such as the migration, development,
and organization of eukaryotic cells \cite{RaEd17}. Roughly speaking, the process of cell polarity is correlated to the reorganization of several
chemicals within a cell and on a cell membrane.  Typically polarization is achieved by the combination of an internal pattern forming system, a response to an external
signal that imposes some directional preference to the pattern, and the amplification of small concentration differences \cite{SkLN05}.

A key step in the polarization process is the direction sensing \cite{ChOt16}, where chemical gradients are detected and amplified. This step proceeds by the transduction of a signal by receptors on the plasma membrane and its adaption by intracellular signaling cascades, which involve the activation and deactivation of specific proteins and the translation of possibly shallow gradients in the outer signal to large amplitude intracellular gradients in protein distributions. Once such polarity of the cell in form a of a spatial asymmetry in chemical concentrations has been established, changes in cell shape and the movement of the cell in the surrounding environment can be initiated.

Polarization is in many instances a dynamic, time-dependent process and a tight regulation of the response to changing environmental conditions is key for many biological functions. One prominent and well-studied example is the chemotaxis of the social amoeba Dictyostelium that migrate to the source of waves of chemoattractant, which exposes the cell to a pulsatile gradient \cite{NIIS14,SYEB14}.

\medskip
Several mathematical models of varying complexity have been suggested to analyze the spatial and temporal processes associated with cell polarization. One of the most popular models is the local excitation, global inhibition (LEGI) mechanism which was suggested in the seminal paper about cell polarization \cite{Mein99}, see also \cite{PaDe99,LeIg02}, and is often part of more comprehensive models \cite{ChOt16}.
\medskip

We focus on a minimal model for the amplification step that has been proposed in \cite{NiRV20}. The significance of the suggested model stems from the fact
that in a suitable parameter regime an asymptotic reduction leads to a generalized obstacle-type problem that allows for a clear and mathematically
tractable characterization of polarized states. In  \cite{NiRV20} we have analyzed stationary states and the onset of polarization.
The present paper continues this analysis by considering the time-dependent problem.
\medskip

The model proposed in \cite{NiRV20} consists of a system of PDEs, motivated by the GTPase cycle model presented in \cite{RaeRoe12,RaeRoe14}.
We consider a protein that can be in an active or an inactive state, where the inactive protein moreover can be bound to the cell membrane or be in a cytosolic state, i.e.~contained in the cells interior. We denote the surface concentration of the active and incative form by $u$ and $v$, respectively, and the volume concentration of the inactive cytosolic state by $w$.
The model has only a few ingredients. It accounts for lateral diffusion on the cell membrane, for diffusion inside the cell,
for activation and deactivation processes on the cell membrane and for attachment to and detachment from the cell membrane. One contribution to the
activation depends on a concentration $c$ of a protein that characterizes an external signal (possibly after a first processing step). This concentration in general may vary with space and time.

Most of these processes are modeled by linear kinetic laws, except for parts of the activation and deactivation processes that need the catalyzation by enzymes and are described by simple Michaelis-Menten type rate laws, see \cite{NiRV20} for more details on the model derivation.
\medskip

To give a mathematical formulation, we represent the cell and its outer cell membrane by a domain $\Omega \subset \mathbb R^3$ and its boundary $\Gamma:=\partial \Omega$.
Moreover we fix a time interval $(0,T)$ of observation, a signal concentration $c:\Gamma\times (0,T)\to\R$, and request that $u,v:\Gamma\times (0,T)$ and $w:\Omega\times (0,T)$ solve the following coupled system of bulk and surface partial differential equations
\begin{align}
	\partial_t u &=\Delta_\Gamma u+\Big(a_1+\frac{a_2u}{a_3+u}+c\Big)v-\frac{a_4 u}{1+u}
	&\text{on $\Gamma \times (0,T)$}\,, \label{eq1}\\
	\partial_t v &=\Delta_\Gamma v-\Big(a_1+\frac{a_2u}{a_3+u}+c\Big)v+\frac{a_4 u}{1+u}-a_5v+a_6w &\text{on } \Gamma \times (0,T)\,,\label{eq2}\\
	\partial_tw&=D\Delta w  &\text{in }\Omega \times (0,T)\,, \label{eq3}\\
	-D\frac{\partial w}{\partial \nu}&=-a_5v+a_6w &\text{on } \Gamma \times (0,T)\,.\label{eq4}
\end{align}
Here $\Delta_\Gamma u$ and $\Delta_\Gamma v$ denote the Laplace-Beltrami
operator on the surface $\Gamma$ and $a_1,\dots,a_6$ are nonnegative constants while $D$ denotes the quotient of the cytosolic diffusion and the lateral membrane diffusion constants, which typically is very large.
\medskip

We complement the system with initial conditions:
\begin{equation}
	u(\cdot,0)=u_\init\,, \quad v(\cdot,0)=v_\init \text{\quad on $\Gamma$}\,, \quad w(\cdot,0)=w_\init \text{\quad in $\Omega$}\,,\label{eq5}
\end{equation}
where $u_\init,v_\init:\Gamma \to [0,\infty)$ and $w_\init:\Omega \to [0,\infty)$ are given nonnegative data.

The system \eqref{eq1}-\eqref{eq4} contains two parts. On the one hand, we have a reaction-diffusion system on the membrane for the variables $u$ and $v$,
with a $w$-dependent source term. On the other hand, there is a diffusion equation for $w$ in the interior of the cell with a nonlinear Robin-type boundary
condition that depends on $u$ and $v$. Solutions of \eqref{eq1}-\eqref{eq5} satisfy the mass conservation property
\begin{equation}
	\int_{\Omega} w(\cdot,t) \,dx+\int_{\Gamma} \big(u(\cdot,t)+v(\cdot,t) \big) \,dS=\int_{\Omega} w_\init \,dx +\int_{\Gamma} \big( u_\init+v_\init \big) \,dS\label{eq6}
\end{equation}
for all $t \in (0,T)$.

In addition to \eqref{eq1}-\eqref{eq4} we will study a reduced system that is obtained in the limit of infinite cytosolic diffusivity, which is motivated by the fact that cytosolic diffusion within the cell is by a factor of hundred larger than the lateral diffusion on the membrane \cite{KhHW00}. In this limit the cytosolic concentration becomes spatially constant and $w=w(t)$ is determined by the total mass conservation, i.e.
\begin{equation}
	 \vert \Omega \vert  w(t)=m-\int_{\Gamma} \big (u(\cdot,t)+ v(\cdot,t) \big )\,dS\,,
	 \label{eq10}
\end{equation}
where $m$ is the total amount of protein.
The reduction for $D=\infty$ leads to a nonlocal reaction-diffusion system on $\Gamma\times (0,T)$, given by \eqref{eq1}, \eqref{eq2} and \eqref{eq10},
complemented by initial conditions for $u$ and $v$. This reduction can be viewed as a kind of shadow system. Such systems have been analyzed intensively in the case of two-variable reaction-diffusion systems in open domains \cite{Keen78,HaSa89,LiNi09}, and in the context of obstacle problems in \cite{Rodr02}.

Under the assumption that the reaction rates $a_4, a_5$ and $a_6$, the diffusion coefficient $D$ and the total mass of proteins are of order $\varepsilon^{-1}$,
we will prove that solutions converge in the \emph{large reaction rate limit} $\eps \to 0$ to solutions of certain reduced systems.
First, we will investigate the limit of infinite cytosolic diffusivity. Taking then the limit $\varepsilon \to 0$, yields the following parabolic obstacle-type problem
\begin{align}
	&\partial_t u -\Delta u  =-a_4(1-g)\xi + \alpha g
	&\mbox{ on } \Gamma \times (0,T)\,, \label{ob1,1}
	\\
	&u\geq 0\,, \quad  u\xi = u \, ,\quad
	0\leq \xi\leq 1  &\mbox{  on } \Gamma \times (0,T)\,,\label{ob1,2}\\
	&u(\cdot,0) = u_0 & \mbox{ on } \Gamma \,, \label{ob1,3}
\end{align}
where $u_0$ is the limit of suitably rescaled versions of $u_\init$ (cf. \eqref{eq15b}), the function  $g :\Gamma \times (0,T) \to (0, 1)$ is given by
\begin{equation}\label{gdef}
	g(x,t)=\frac{c(x,t)}{c(x,t)+a_5}\,,
\end{equation}
and  $\alpha:(0,T)\to\R$ only depends on time and is determined by a solvability condition for \eqref{ob1,1}, see \eqref{alpha}. This function $\alpha$ plays the role a Lagrange multiplier associated to the mass conservation property
\begin{equation*}
  \int_\Gamma u(\cdot,t)\,dS = \int_\Gamma u_0\,dS \quad
  \text{ for all }t\in (0,T)\,,
\end{equation*}
that is satisfied in the limit.

In the case $D<\infty$ equation \eqref{ob1,1} changes and we obtain the system
\begin{align}
  \partial_t u&=\Delta u-a_4(1-g)\xi+a_6 g w \quad &\mbox{ on } \Gamma \times (0,T)\,, \label{ob2,1}\\
  0&=\Delta w \quad\text{ in } \Omega_T,  \quad& \text{ on } \Gamma \times (0,T) \,, \label{ob2,2a}
  \\
  D\frac{\partial w}{\partial n} &= a_4(1-g)\xi-a_6gw  \quad&\text{ on } \Gamma \times (0,T) \,, \label{ob2,2}
  \\
	u &\geq 0\,, \quad  u\xi = u \, ,\quad
	0\leq \xi\leq 1  \quad &\mbox{  on } \Gamma \times (0,T)\,,\label{ob2,3}\\
	u(\cdot,0) &= u_0 &\mbox{ on } \Gamma \,. \label{ob2,4}
\end{align}
The analogy to $D=\infty$ is even more apparent if one expresses $w$ as a nonlocal operator of $u$.
A particularly convenient form is presented in Proposition \ref{P.equivalence}.

Stationary solutions of model \eqref{eq1}-\eqref{eq4} and the corresponding scaling limits  have already been studied in \cite{NiRV20}.
In particular, in addition to well-posedness, the onset of polarization is studied in \cite{NiRV20} for sufficiently small
(rescaled) mass of protein. The goal of the present paper is to complement this analysis. The main contributions are a rigorous  justification
of the asymptotic reduction, the well-posedness of the evolutionary obstacle-type problem that we obtain in the limit, an $L^1$-contraction property of solutions, and the global
stability of steady states.

\bigskip
Parabolic obstacle problems appear in various applications and have have been studied intensively over the past decades \cite{Frie88}. For example, the one-phase Stefan problem can be written as a parabolic obstacle problem by a suitable  transformation that was first proposed by Duvaut \cite{Duva73}.
In the context of fluid flows in porous media the Baiocchi transform \cite{Baio72} also leads to an obstacle problem.
Obstacle problems belong to a class of free boundary problems that can be formulated as variational inequalities, i.e.~inequalities for bilinear functionals which are satisfied for functions $u$ and test functions in a space satisfying inequalities of the form $u\geq\psi$.
Alternatively, under some regularity assumptions it is possible to reformulate the same class of free boundary problems as PDEs in which an unknown function $\xi$ satisfies an inequality almost everywhere in the set in which the PDEs are solved. Both formulations can be found for example in \cite{KiSt00,Rodr87}.
The equivalence between both approaches can be seen using the so-called Stampacchia Lemma \cite[Section 5:3, Theorem 5:4.3]{Rodr87}.
In this paper we will only use the second approach. Therefore, in addition to the unknown $u$ we must determine also an auxiliary function $\xi\in [0,1]$ such that $\xi=1$ in $\{u>0\}$.

\medskip
The connection of our limit problems with the parabolic obstacle problem is best seen for the reduced model \eqref{ob1,1}-\eqref{ob1,2}. In Remark \ref{2.Rxi} we derive the following characterization of solutions,
\begin{equation}
 H(u):= \partial_t u - \Delta u + \big(a_4(1-g) - \alpha g\big) = \big(a_4(1-g) - \alpha g\big)_+\Chi_{\{u=0\}},\qquad u\geq 0,
 \label{eq:betaform}
\end{equation}
with $\alpha=\alpha(t)$ given as a nonlocal function of $u$, more precisely
\begin{equation*}
  \alpha(t)  =\frac{a_4\int_{\{u(\cdot,t)>0\}}\big(1-g(\cdot,t)\big)\,dS}{\int_{\{u(\cdot,t)>0\}}g(\cdot,t)\,dS},
\end{equation*}
see \eqref{alpha} and \eqref{alphaplus}.
In the formulation \eqref{eq:betaform} the problem corresponds to the classical parabolic obstacle model, where $a_4(1-g) - \alpha g$ is replaced by some given function $f$ independent of $u$. The problem  \eqref{eq:betaform} can be written as
\begin{equation*}
 uH(u) = 0,\qquad H(u) \geq 0,\qquad u\geq 0,
\end{equation*}
and can be expressed as a variational inequality, see for example \cite[Section II.9.1]{Lion69}.

One of the features of the free boundary problems considered in this paper is the presence of some terms in the equations that depend in a non-local way on the solution $u$ itself.
In the case of problem \eqref{ob1,1}-\eqref{ob1,3}, the non-locality is introduced by the dependence of $\alpha$ in \eqref{eq:betaform} on the positivity set  $\{u>0\}$.
For the bulk-surface problem \eqref{ob2,1}-\eqref{ob2,4}
the non-local dependence takes place through the function $w$ which solves the elliptic problem \eqref{ob2,2a}, \eqref{ob2,2}.
We remark that free boundary problems containing dependences on the positivity set of the solution itself (i.e.~$\{
u>0\}$) have been considered in \cite{Rodr02}.

Several of the technical difficulties that we need to address in this paper are due to the fact that the function $\alpha$ changes in a discontinuous manner if the positivity set $\{u(\cdot,t)>0\}$ changes discontinuously in time.
However, to prove that $\{u(\cdot,t)>0\}$ changes continuously in time is not an easy task and we expect that jumps of this set are possible in some situations.
We will address the continuity properties of $\{u(\cdot,t)>0\}$ and $\alpha$ in future work, but remark here that possible jumps of the functions $t\mapsto \{u(\cdot,t)>0\}$ and $t\mapsto\alpha(t)$ are the main reason for several of the most technical points of this paper.

\medskip
Compared to \cite{NiRV20} the main novelty of this paper is to introduce some monotonicity formulas which allow us to prove uniqueness of solutions and also uniqueness and stability of steady states of the problems \eqref{ob1,1}-\eqref{ob1,3} and \eqref{ob2,1}-\eqref{ob2,3}.

Uniqueness of the steady states associated to the problem \eqref{ob1,1}-\eqref{ob1,3} has been proved in \cite{NiRV20} using a completely different approach.
Similar uniqueness results have been obtained in \cite{NiRV20} for the stationary states of \eqref{ob2,1}-\eqref{ob2,3} in the particular case in which the domain $\Omega$ is a ball.
The monotonicity formulas introduced in this paper (cf. Sections \ref{3} and \ref{4}) imply that the evolution semigroup associated to the problems \eqref{ob1,1}-\eqref{ob1,3} and \eqref{ob2,1}-\eqref{ob2,3} is contractive in the $L^{1}$ norm.
It is worth to remark that the existence of these monotonicity formulas rely on a delicate balance of the terms $-a_4(1-g)\xi$ and $\alpha g,\ a_5 gw$ in \eqref{ob1,1}, \eqref{ob2,1}.
The term $-a_{4}(1-g)\xi$ has a stabilizing effect, which is similar to the analogous term arising in the study of the reformulation of the one-phase Stefan problem due to Duvaut \cite{Duva73}.
On the other hand the terms $\alpha g,\ a_5 gw$ in \eqref{ob1,1}, \eqref{ob2,1} depend on functions determined as a non-local functional of $u$ (namely $\alpha$ and $w$ respectively).
These terms have a destabilizing effect on the solutions of the problems \eqref{ob1,1}-\eqref{ob1,3} and \eqref{ob2,1}-\eqref{ob2,3}, but some cancellations between the contributions of both terms in the derivative of the $L^{1}$ norm of the difference of two solutions of these problems yield an overall stabilizing effect.

\bigskip
The plan of this paper is the following.
Section \ref{2} is devoted to establishing the convergence of solutions in the fast reaction limit to the limiting obstacle-type problems.
In Section \ref{2.1} we will first investigate the case of infinite cytosolic diffusion $D=\infty$, introduce a suitable rescaled system \eqref{eq13}-\eqref{eq15b} and prove the convergence to \eqref{ob1,1}-\eqref{ob1,3} (cf. Theorem \ref{T.conv1}).
In Section \ref{2.2} we consider the analogous problem for finite cytosolic diffusion coefficients $D$.
We derive in a scaling limit analogous to the case $D=\infty$ the generalized obstacle-type problem \eqref{ob2,1}-\eqref{ob2,4} in Theorem \ref{T.conv2}.
Section \ref{3} focuses on the case $D=\infty$.
In Section \ref{3.1} we justify an $L^1$-contraction property and the uniqueness of solutions of problem \eqref{ob1,1}-\eqref{ob1,3} (Theorem \ref{T.uniqueness1}) while in Section \ref{3.2} we will show the global stability
of the steady states (Theorem \ref{T.stability1}).
In Section \ref{4} we study the reduced model for finite cytosolic diffusion $D< \infty$. We prove an $L^1$-contraction property and the uniqueness of solutions of problem \eqref{ob2,1}-\eqref{ob2,4} in
Section \ref{4.1}, see Theorem \ref{T.uniqueness2}. We also include a monotonicity property and a uniqueness result for solutions of the stationary problem in Theorem \ref{thm:monotonicity}.
This improves the corresponding result from \cite{NiRV20} that was only shown for $\Omega=B(0,1)$ there. Along the lines of Section \ref{3.2}, we will further show in
Section \ref{4.2}, Theorem \ref{T.stability2}, that steady states are globally stable.

\subsection{Notation and Assumptions}

\textbf{Notations:} For a set $\Omega \subset \mathbb R^3$ we denote by $\vert \Omega \vert =\mathcal L^3(\Omega)$ the Lebesgue measure. For a surface $\Gamma \subset \mathbb R^3$ we
denote by $\vert \Gamma \vert=\mathcal H^2(\Gamma)$ its area (i.e. the 2-dimensional Hausdorff measure) and by $\int_{\Gamma}\cdot\; dS$ the corresponding surface integral.

For the sake of convenience, $\Omega_T$ and $\Gamma_T$ stand for $\Omega \times (0,T)$ and  $\Gamma \times (0,T)$ respectively.
For the Laplace-Beltrami operator on $\Gamma$ we just write $\Delta$ instead of $\Delta_\Gamma$ if there is no reason for confusion.

We denote the usual Sobolev spaces by $W^{k,p}(U)$ and the parabolic Sobolev spaces by $W^{k,k/2}_p(U_T)$, where $U=\Omega$ or $U=\Gamma$, $k\in\N_0$, $1\leq p\leq\infty$. The Hölder and parabolic Hölder spaces are denoted by $C^\alpha(U)$ and $C^{\alpha,\alpha/2}(U_T)$, respectively, for $0<\alpha<1$.
The weak parabolic solution spaces are denotes by $V_2(U_T):=L^2(0,T;H^1(U))\cap H^1(0,T;H^1(U)^*)$.

\medskip
\textbf{Assumptions:} Let $\Omega \subset \mathbb R^3$ be an open, bounded, connected set with $C^3-$regular boundary
$\Gamma=\partial \Omega$. Assume $a_1,a_2 \ge 0$, $a_3, a_4,a_5,a_6 > 0$ and $D \ge 1$ and that $c :\Gamma_T \to \mathbb R_{+}$ is smooth and that there exists $c_0 > 0$ with
\begin{equation}
  c(x,t) \ge c_0 >0 \quad \text{ for all }(x,t) \in \Gamma_T\,. \label{assum}
\end{equation}

\section{The fast reaction limit}
\label{2}
\subsection{Convergence to a parabolic obstacle-type problem for $D= \infty$ \label{2.1}}

In this section we consider the case of infinite cytosolic diffusion coefficient, that is we consider solutions to \eqref{eq1}, \eqref{eq2} and \eqref{eq5} together with \eqref{eq10}.
It follows from \cite{HaRo18} that for given $m > 0$ and for  nonnegative data $u_\init,v_\init\in L^2(\Gamma)$ with $\int_{\Gamma} (u_\init+v_\init)\,dS \leq m$ there exists
a nonnegative solution  $(u,v,w)$ with
$u,v\in \Xspace$ and $w\in W^{1,\infty}(0,T)$. In fact, although the analysis in \cite{HaRo18} does not consider nonconstant $c$ all arguments easily carry over to the present case.

\medskip
Our goal in this section is to consider a suitable scaling limit of the system \eqref{eq1}, \eqref{eq2}, \eqref{eq5} and \eqref{eq10}.
More precisely, for small $\eps>0$  we introduce the following rescalings
\begin{equation}
	a_4 = \frac{\hat a_4}\eps,\quad a_5 = \frac{\hat a_5}\eps,\quad a_6 = \frac{\hat a_6}\eps, \quad
	c=\frac{\hat c}{\eps} \quad \text{and} \quad m = \frac{\hat m}\eps \label{eq12}
\end{equation}
with $\hat a_4,\hat a_5,\hat a_6, \hat c$ and $\hat m$ being positive and of order one. We denote the corresponding solutions by $u_{\eps}, v_{\eps}$ and $w_{\eps}$ and let $U_\eps := \eps u_\eps$.
Dropping the hats again, we can  rewrite \eqref{eq1}, \eqref{eq2}, \eqref{eq5} and \eqref{eq10} as
\begin{align}
	\partial_t U_\eps &=\Delta U_\eps+\Big(\eps a_1+\frac{\eps a_2U_\eps}{\eps a_3+U_\eps}+c\Big)v_\eps-
	\frac{a_4 U_\eps}{\eps+U_\eps } \qquad  \qquad \text{ on } \Gamma_T\,,\label{eq13}\\
	\eps\partial_t v_\eps &=\eps\Delta v_\eps-\Big(\eps a_1+\frac{\eps a_2U_\eps}{\eps a_3+U_\eps}+
	c\Big)v_\eps+\frac{a_4 U_\eps}{\eps+U_\eps}-a_5v_\eps+a_6w_\eps \qquad \text{ on } \Gamma_T\,,\label{eq14}\\
	\eps \vert \Omega \vert w_\eps(t)&= m-\int_{\Gamma} ( U_\eps(x,t)+\eps v_\eps(x,t) )\,dS \qquad \qquad  \mbox{ for a.a. }  t \in (0,T)\,.\label{eq15}
\end{align}
For given nonnegative, smooth functions $U^\eps_0,v^\eps_0:\Gamma\to\R$ with $\int_\Gamma \big(U^\eps_0 +\eps v^\eps_0\big) \leq m$ we prescribe the initial conditions
\begin{equation}
	U_\eps(\cdot,0)=U^\eps_0\,, \quad v_\eps(\cdot,0)=v^\eps_0 \quad\text{ on }\Gamma\,.
	\label{eq5eps}
\end{equation}
In order to obtain a nontrivial limit, we assume for the initial data that
\begin{equation}\label{eq15b}
	U^\eps_0 \to u_0 \quad \mbox{ in } L^2(\Gamma) \text{ as } \eps \to 0\,,\quad
  \sup_{\eps> 0}\Big[\int_\Gamma |v^\eps_0|^2\,dS + \frac{1}{\eps}\Big(m - \int_\Gamma \big(U^\eps_0 +\eps v^\eps_0\big)\,dS\Big)\Big]\leq C\,
\end{equation}
for some $u_0\in L^2(\Gamma)$ with $\int_{\Gamma} u_0\,dS=m$ and some $C>0$.

\medskip
We first prove some uniform estimates.

\begin{theorem}\label{T.bounds1}
	For any  nonnegative solution $(U_\eps, v_\eps, w_\eps)$ of \eqref{eq13}-\eqref{eq5eps} we have
	\begin{equation}
		\| U_\eps \|_{\Xspace}+
		\| v_\eps \|_{L^\infty(0,T;L^2(\Gamma))}+\| w_\eps \|_{L^\infty(0,T)}
		\leq C\,, \label{eq55a}
	\end{equation}
	where here and in the following $C$ denotes a constant that depends on the data of the problem but not on $\eps$.
\end{theorem}

\begin{proof}
By virtue of \eqref{eq14} we compute
\begin{align}
	\frac{d}{dt} \int_{\Gamma} \frac{\eps a_5 v^2_\eps}{2}\,dS
	&= - \int_\Gamma  \eps a_5 {\vert \nabla v_\eps \vert }^2\,dS -\int_{\Gamma} a_5 \Big(\eps a_1+\frac{\eps a_2U_\eps}{\eps a_3+U_\eps}+c\Big)v_\eps^2\,dS
  \notag\\
	&\qquad
	+\int_{\Gamma} \frac{a_4 a_5 U_\eps v_\eps}{\eps+U_\eps}\,dS
  -\int_{\Gamma} \big( ( a_5 v_\eps)^2- a_5 a_6 v_\eps w_\eps \big )\,dS\,.
  \label{eq2,1}
\end{align}
We observe that \eqref{eq13}-\eqref{eq15} imply that
\begin{equation}
	\eps |\Omega|\frac{d}{dt}w_\eps = \int_\Gamma (a_5v_\eps -a_6 w_\eps)\,dS\,,\qquad
  \eps |\Omega| w_\eps(0) = m - \int_\Gamma \big(U^\eps_0 +\eps v^\eps_0\big)\,dS
	\label{eq:15ode}
\end{equation}
and obtain
\begin{equation}
	\eps |\Omega| a_6 \frac{d}{dt}\frac{1}{2}w_\eps^2 =
	a_5a_6\int_\Gamma v_\eps w_\eps \,dS - a_6^2 |\Gamma| w_\eps^2.
  \label{eq2,2}
\end{equation}
Taking the sum of \eqref{eq2,1} and \eqref{eq2,2} and using $c \geq c_0 >0 $ yields
\begin{align*}
	&\frac{d}{dt} \Big(\int_{\Gamma} \frac{\eps a_5 v^2_\eps}{2}\,dS + \eps |\Omega| a_6 \frac{1}{2}w_\eps^2\Big)
	+\int_\Gamma  \eps a_5 {\vert \nabla v_\eps \vert }^2\,dS +c_0 \int_{\Gamma} a_5 v_\eps^2\,dS \\
	&\qquad
	\leq \int_\Gamma a_4a_5 v_\eps\,dS  -\int_\Gamma \big( a_5^2 v_\eps^2- 2a_5 a_6 v_\eps w_\eps +a_6^2 w_\eps^2\big )\,dS\\
  &\qquad \leq \int_\Gamma \frac{c_0a_5}{4}v_\eps^2 + \frac{1}{c_0}a_4^2a_5  + \Big(\frac{1}{\delta}-1\Big)a_5^2v_\eps^2 - (1-\delta)a_6^2w_\eps^2\,dS\,,
\end{align*}
where we have used Young's inequality and where $\delta>0$ is arbitrary.

We next choose $\delta<1$ sufficiently close to one such that $(\frac{1}{\delta}-1\big)a_5<\frac{1}{4}c_0$ and obtain
\begin{equation*}
  \eps \frac{d}{dt} \Big(\int_\Gamma v_\eps^2\,dS + w_\eps^2\Big)
  \leq C - \Big(\int_\Gamma v_\eps^2\,dS + w_\eps^2\Big).
\end{equation*}
Using \eqref{eq15b} we deduce
\begin{equation}
	\int_{\Gamma} v_\eps^2(\cdot,t)\,dS + w_\eps^2(t) \leq C\quad\text{ for all }0\leq t\leq T.
	\label{eq:L2bound_v}
\end{equation}
This implies the required bounds for $v_\eps,w_\eps$.

Furthermore, by these estimates the reaction-terms on the right-hand side of \eqref{eq13} are uniformly bounded in $L^2(\Gamma_T)$.
Parabolic $L^2$ theory, see \cite[Section 7.1]{Evan10}, and \eqref{eq15b} imply the uniform bound for $U_\eps$, which finishes the proof of \eqref{eq55a}.
\end{proof}

\begin{theorem}\label{T.conv1}
Suppose that $\{(U_\eps ,v_\eps , w_\eps )\}_{\eps >0}$ is a family of nonnegative solutions of \eqref{eq13}-\eqref{eq5eps} and assume \eqref{eq15b}.
Then there exist a subsequence $\eps \to 0$, a nonnegative function $u \in \Xspace$ and a measurable function $\xi$ such that
\begin{align}
	U_{\eps} &\weakto u \quad\text{ in }\Xspace \,,
	\label{eq:T.conv1-1}\\
	\frac{U_{\eps}}{U_{\eps}+{\eps}}&\overset * \rightharpoonup \xi \quad \text{ weakly* in }  L^{\infty}(\Gamma_T) \label{eq:T.conv1-2}
\end{align}
as $\eps \to 0$.
Moreover, $u\in W^{2,1}_p(\Gamma\times (\delta,T))$ for any $\delta>0$, $1\leq p<\infty$, with
\begin{equation}
	\|u\|_{W^{2,1}_p(\Gamma\times (\delta,T))} \leq C(p,\delta,T),
	\label{eq2regu}
\end{equation}
and  $\int_{\Gamma} u(\cdot,t)\,dS = m$ holds for  all $t\in [0,T]$.

Finally, there exists a nonnegative function $\alpha\in L^\infty(0,T)$
such that \eqref{ob1,1}-\eqref{ob1,3} are satisfied.
\end{theorem}

\begin{proof}
	By Theorem \ref{T.bounds1} there exists a subsequence $\eps\to 0$ (not relabeled) and functions $u\in \Xspace$, $v\in L^\infty(0,T;L^2(\Gamma))$, $w\in L^\infty(0,T)$ and $\xi\in L^\infty(0,T)$ such that
	\begin{align*}
		U_{\eps} &\weakto u \quad\text{ in }\Xspace,\\
		v_{\eps} &\weakstarto v  \quad\text{ in }L^\infty(0,T;L^2(\Gamma)),\\
		w_{\eps} &\weakstarto w \quad\text{ in }L^\infty(0,T),\\
		\frac{U_{\eps}}{U_{\eps}+{\eps}}&\weakstarto \xi \quad \text{ in }  L^{\infty}(\Gamma_T).
	\end{align*}
	By the Aubin-Lion's compactness Lemma \cite{Aubi63,Roub13} we also have
	\begin{equation}
		U_{\eps} \to u \quad\text{ in } L^2(\Gamma_T).
		\label{eq:T.conv1-3}
	\end{equation}
	With these convergence properties we can pass to the limit in the weak form of \eqref{eq13} and conclude that for any $\phi\in C^1_c(\Gamma\times [0,T))$ it holds
	\begin{equation}
		\int_{\Gamma_T} \partial_t\phi (u-u_0) \,dS\,dt= \int_{\Gamma_T} \big(\nabla\phi\cdot\nabla u -\phi\big(cv-a_4\xi(\cdot,t)\big)\big)\,dS\,dt\,.
	\end{equation}
   In particular, \eqref{ob1,1} holds in $H^{-1}(\Gamma)$ for almost all $t\in (0,T)$.
  Moreover, by \cite[Theorem 25.5]{Wlok87} we have $u\in C^0([0,T];L^2(\Gamma))$, and $u(\cdot,0)=u_0$ holds in the sense that  $u_0=\lim_{t\searrow 0}u(\cdot,t)$ in $L^2(\Gamma)$.

	Let $\phi \in C^{2,1}_c(\Gamma_T)$ be an arbitrary test function. Multiplying \eqref{eq14} by $\phi$ and integrating over $\Gamma_T$ we deduce, after integrating by parts, that
	\begin{align}
		-\eps\int_0^T &\langle v_\eps,\partial_t \phi \rangle\,dt =\eps \int_0^T \int_{\Gamma} v_\eps \Delta \phi\,dS\,dt
		-\int_0^T \int_{\Gamma}\Big(\eps a_1+\frac{\eps a_2U_{\eps}}{\eps a_3+U_{\eps}}+c\Big)v_\eps \phi \,dS\,dt \nonumber\\
		&+\int_0^T \int_{\Gamma}\frac{a_4 U_{\eps}\phi}{\eps+U_{\eps}}\,dS\,dt
		-\int_0^T \int_{\Gamma}a_5v_\eps \phi\,dS\,dt
		+ \int_0^T a_6 w_\eps \int_\Gamma \phi\,dS\,dt \,.
		\label{eq26}
	\end{align}
	Taking the limit in $(\ref{eq26})$ we obtain
	\begin{equation*}
		0 = -\int_0^T \int_{\Gamma} (c+a_5)v\phi\,dS\,dt -a_4\int_0^T \int_{\Gamma} \xi\phi\,dS\,dt +\int_0^T a_6 w \int_\Gamma \phi\,dS\,dt\,,\
	\end{equation*}
	hence
	\begin{equation}
		0 = -(c+a_5)v -a_4\xi +a_6 w \quad\text{ a.e.~in }\Gamma_T.
		\label{eq27}
	\end{equation}
	Similarly we deduce from \eqref{eq15} that
	\begin{equation}
		0 = m- \int_\Gamma u(\cdot,t) \,dS\quad\text{ in } (0,T)
		\label{eq28}
	\end{equation}
	and from \eqref{eq:15ode}
	\begin{equation}
		0 = \int_\Gamma( a_5 v(\cdot,t)-a_6w(t)) \,dS\quad\text{ a.e.~in } (0,T).
		\label{eq29}
	\end{equation}
	Finally, we define
	\begin{equation*}
		\alpha(t)=\frac{a_5}{\vert \Gamma \vert} \int_{\Gamma}v(\cdot,t)\,dS
	\end{equation*}
	such that \eqref{eq27} and  \eqref{eq29} imply
	\begin{equation}
		v=- \frac{a_4}{c+a_5}\xi+\frac{\alpha}{c+a_5} \qquad \mbox{ a.e. in } \Gamma_T\,. \label{eq31}
	\end{equation}
	Due to the boundedness of $g$ and $\xi$ we can apply parabolic $W^{2,1}_p$-regularity theory to \eqref{ob1,1}. In fact, fix arbitrary $\delta>0$ and $p\geq 1$.
  Choose a smooth cut-off function $\eta\in C^\infty_c((\frac{\delta}{2},T])$, $\eta=1$ in $[\delta,T]$ and use a smooth partition of unity for $\Gamma$ subordinate to a covering of $\Gamma$ by parametrized surface patches.
  In local coordinates we obtain that $\eta u$ solves a parabolic equation with bounded continuous coefficients, hence \cite[Theorem IV.9.1]{LaSU68} yields the $W^{2,1}_p$-regularity of $\eta u$ in local coordinates, with an estimate for the corresponding norms only depending on the data.
  Using the compactness of $\Gamma$ we finally deduce the $W^{2,1}_p$-regularity of $\eta u$, hence \eqref{eq2regu} holds.

	From \eqref{eq:T.conv1-3} we obtain for any
	test function $\phi \in C^0(\Gamma \times [0,T])$ that
	\begin{align}\label{31b}
		\int_0^T \int_{\Gamma} \phi (\xi u{-}u)\,dS\,dt =
		\lim_{j \to \infty}\int_0^T \int_{\Gamma} \phi \Big (\frac{U_{\eps}}{U_{\eps}+{\eps}} {-}1 \Big)U_{\eps}\,dS\,dt
		=-\lim_{j \to \infty} \int_0^T \int_{\Gamma} \frac{\eps \phi U_{\eps}}{\eps+U_{\eps}}\,dS\,dt=0\,,
	\end{align}
	which implies $\xi u=u$.
\end{proof}

\begin{remark}\label{2.Rxi}
	By Stampacchia's Lemma \cite[Theorem 4.4]{EvGa15} and the $W^{2,1}_p(\Gamma\times (\delta,T))$-regularity of $u$ for any $\delta>0$ one obtains $\partial_t u=\Delta u = 0$ almost everywhere in $\{u = 0\}$. In fact, we can apply the lemma to $W^{1,p}(\Gamma_T)$ and obtain the claim for $\partial_tu$, and then to $W^{2,p}(\Gamma)$ for almost all $t$ to obtain the corresponding property for $\Delta u$. This in particular yields the representation formula
	\begin{equation}
		\xi(\cdot,t)=
		\begin{cases}
			1 &\text{ a.e~in } \{u(\cdot,t)>0 \}\\
			\frac{\alpha(t) g(\cdot,t)}{1-g(\cdot,t)}&\text{ a.e~in } \{u(\cdot,t)=0 \}
		\end{cases} \label{xi}
	\end{equation}
	for almost all $t\in (0,T)$. By $\xi\leq 1$ we deduce that
  \begin{equation}
    \alpha g\leq 1-g \quad\text{ almost everywhere in }\{u=0\}.
    \label{eq:cond-u=0}
  \end{equation}
	Moreover, by an integration of \eqref{ob1,1} over $\Gamma$ and by \eqref{xi} we deduce that
	\begin{equation}
		\alpha(t) = \frac{\int_\Gamma (1-g)(\cdot,t)\xi(\cdot,t)\,dS}{\int_\Gamma g\,dS}
		= \frac{\int_{\{u(\cdot,t)>0\}} (1-g)(\cdot,t)\,dS}{\int_{\{u(\cdot,t)>0\}} g(\cdot,t)\,dS}
		\label{alpha}
	\end{equation}
	for almost all $t\in (0,T)$. Note that the second equality in \eqref{alpha} shows that $\alpha$ is already determined by $u$ and the data. Similarly, for any measurable set $A\supset \{u(\cdot,t)>0\}$ we deduce that
  \begin{equation}
		\alpha(t) = \frac{\int_A (1-g)(\cdot,t)\xi(\cdot,t)\,dS}{\int_A g(\cdot,t)\,dS}
		\label{alphaA}
	\end{equation}
  holds for almost all $t\in (0,T)$.

  We derive a further characterization of solutions. By the properties obtained so far we deduce from \eqref{ob1,1}, $\xi\leq 1$ and \eqref{alpha} that
  \begin{align}
    \partial_t u -\Delta u  
    &= -a_4(1-g) + \alpha g + \big(a_4(1-g) - \alpha g\big)_+\Chi_{\{u=0\}}.
    \label{alphaplus}
  \end{align}
  Vice versa, this equation implies \eqref{ob1,1}, with $\xi$ as \eqref{xi}, and the conditions on $\xi$ in \eqref{ob1,2}. 
\end{remark}

\subsection{Convergence to a parabolic obstacle-type problem for $ D < \infty$ \label{2.2}}

We now consider the case of finite cytosolic diffusion $D < \infty$.
In \cite{HaRo18} it is proved that also in this case the system \eqref{eq1}-\eqref{eq5} has a unique nonnegative solution $(u,v,w)$ with $u,v \in \Xspace$ and $w \in \XspaceOm$, provided that the initial data are such that $u^\eps_0,v^\eps_0 \in L^2(\Gamma)$ and $w^\eps_0 \in L^2(\Omega)$.
Again, this result first only covers the case of constant $c$. The proof, however, carries over to the present case.

For finite $D$ we use a similar rescaling of the general model
\eqref{eq1}-\eqref{eq5} as in the previous subsection but consider in addition to \eqref{eq12} that $D$ becomes large with $\eps \to 0$, more precisely
$D= \frac{\hat D}\eps $. This yields, after dropping the hats, the system
\begin{align}
	\partial_t U_\eps &=\Delta U_\eps+\Big(\eps a_1+\frac{\eps a_2U_\eps}{\eps a_3+U_\eps}+c\Big)v_\eps-
	\frac{a_4 U_\eps}{\eps+U_\eps } & \text{ \quad on $\Gamma_T$}\,,\label{eq50}\\
	\eps\partial_t v_\eps &=\eps\Delta v_\eps-\Big(\eps a_1+
	\frac{\eps a_2U_\eps}{\eps a_3+U_\eps}+c\Big)v_\eps+\frac{a_4 U_\eps}{\eps+U_\eps}-
	a_5v_\eps+a_6w_\eps & \text{ on } \Gamma_T\,,\label{eq51}\\
	\eps \partial_t w_\eps&=D\Delta w_\eps& \text{ on } \Omega_T\,,\label{eq52}\\
	-D\frac{\partial w_\eps}{\partial n}&=-a_5 v_\eps + a_6 w_\eps & \text{  on } \Gamma_T\,,\label{eq53}\\
	U_{\eps}(\cdot,0)&=U^\eps_0\,, \qquad v_{\eps}(\cdot,0)=v^\eps_0\,,\quad w_{\eps}(\cdot,0)=w^\eps_0\,,\label{eq53b}
\end{align}
where $\int_\Gamma \big(U^\eps_0 + \eps v^\eps_0\big)\,dS + \int_\Omega \eps w^\eps_0\,dx = m$.

Similarly as in Section \ref{2.1} we assume that for some $u_0\in L^2(\Gamma)$ with $\int_\Gamma u_0=m$ and some $C>0$ we have
\begin{equation}
	U^\eps_0\to u_0\quad\text{ in }L^2(\Gamma),\qquad
	\sup_{\eps> 0}\Big(\int_\Gamma |v^\eps_0|^2\,dS +
  \int_\Omega |w^\eps_0|^2\,dx\Big) \leq C.
	\label{eq:3init}
\end{equation}
We recall that a solution conserves the mass, that is
\begin{equation}
	\int_{\Omega} \eps w_\eps(\cdot,t)\,dx + \int_{\Gamma} \big(U_\eps(\cdot,t)+\eps v_\eps(\cdot,t)\big)\,dS=m \quad  \mbox{ for all }  t\in (0,T)\,. \label{eq54}
\end{equation}
We first prove some uniform bounds.

\begin{theorem}\label{T.bounds2}
	For any  nonnegative solution $(U_\eps, v_\eps, w_\eps)$ of \eqref{eq50}-\eqref{eq:3init} we have
	\begin{equation}
		\| U_\eps \|_{\Xspace}+
		\| v_\eps \|_{L^\infty(0,T;L^2(\Gamma))}+\| w_\eps \|_{L^\infty(0,T;L^2(\Gamma)))}
    +\| w_\eps \|_{L^2(0,T;H^1(\Omega))}
		\leq C\,. \label{eq55}
	\end{equation}
\end{theorem}

\begin{proof}
	As in the proof of Theorem \ref{T.bounds1} we test \eqref{eq51} with $a_5v_\eps$ and obtain
	\begin{align*}
		\frac{d}{dt} \int_{\Gamma} \frac{\eps a_5 v^2_\eps}{2}\,dS
		&= - \int_\Gamma  \eps a_5 {\vert \nabla v_\eps \vert }^2\,dS -\int_{\Gamma} a_5 \Big(\eps a_1+\frac{\eps a_2U_\eps}{\eps a_3+U_\eps}+c\Big)v_\eps^2\,dS
		+ \int_{\Gamma} \frac{a_4 a_5 U_\eps v_\eps}{\eps+U_\eps}\,dS\\
		&\qquad
		-\int_{\Gamma} \big( ( a_5 v_\eps)^2- a_5 a_6 v_\eps w_\eps \big )\,dS\,.
	\end{align*}
	By virtue of \eqref{eq52} and \eqref{eq53} we  compute
	\begin{equation*}
		\frac{d}{dt} \int_{\Omega} \frac{\eps a_6 w^2_\eps}{2}\,dx =
		- \int_{\Omega} a_6 D {\vert \nabla w_\eps \vert }^2\,dx
		+\int_{\Gamma}( a_5a_6 w_\eps v_\eps - a_6^2w_\eps^2)\,dS\,.
	\end{equation*}
	Combining both inequalities and using $c \geq c_0 >0 $ and $ \frac{U_\eps}{\eps+U_\eps} \leq 1$ implies
	\begin{align}
		&\frac{d}{dt} \Big ( \int_{\Gamma} \frac{\eps a_5 v^2_\eps}{2}\,dS +  \int_{\Omega} \frac{\eps a_6 w^2_\eps}{2}\,dx\Big )
		+ \int_\Gamma  \eps a_5 {\vert \nabla v_\eps \vert }^2\,dS + \int_{\Omega} a_6 D {\vert \nabla w_\eps \vert }^2\,dx
		+ \int_\Gamma a_5 c_0  v^2_\eps\,dS \notag\\
		&\qquad\qquad \leq \int_\Gamma a_4 a_5 v_\eps \,dS - \int_{\Gamma} \big ( a_5^2 v_\eps^2-2a_5a_6 v_\eps w_\eps +a_6^2 w_\eps^2\big)\,dS \notag\\
    &\qquad\qquad \leq \frac{1}{2}\int_\Gamma a_5 c_0  v^2_\eps\,dS + C -c\int_\Gamma w_\eps^2 \,dS, \label{eq2001}
	\end{align}
  where in the last step we have used a Youngs inequality as in the derivation of \eqref{eq:L2bound_v}, and where $C,c>0$ only depend on the data.
  Next, applying  Poincar\'e's  inequality
  for functions with mean value zero on the boundary, we deduce
  \begin{align*}
    \int_\Omega w_\eps^2 \,dx &
    \leq 2\int_\Gamma \Big| w_\eps -\frac{1}{\vert \Gamma \vert} \int_{\Gamma} w_\eps\,dS \Big|^2\,dS
    +2\frac{|\Omega|}{{\vert \Gamma \vert}^2} \Big ( \int_{\Gamma} w_\eps \,dS\Big )^2 \\
    & \leq C\Big(\int_\Omega | \nabla w_\eps |^2\,dx +  \int_{\Gamma} {w_\eps}^2\,dS \Big)
  \end{align*}
  and therefore we obtain from \eqref{eq2001}
  \begin{equation*}
    \eps\frac{d}{dt}\Big( \int_{\Gamma}  v_\eps^2\,dS +  \int_{\Omega} w_\eps^2\,dx\Big) \leq C -  c\Big( \int_{\Gamma}  v_\eps^2\,dS +  \int_{\Omega} w_\eps^2\,dx\Big)\,.
  \end{equation*}
  Hence \eqref{eq:3init} yields a uniform bound for $\|v_\eps\|_{L^\infty(0,T;L^2(\Gamma))}$ and  $\|w_\eps\|_{L^\infty(0,T;L^2(\Omega))}$.

  By an integration of \eqref{eq2001} we in addition obtain
	\begin{align*}
		\int_0^T \int_{\Omega} D |\nabla w_\eps|^2 \,dx\,dt
		 \leq C\,.
	\end{align*}
	Finally, weak solution theory for parabolic equations (see \cite[Section 7.1]{Evan10}), implies
	a uniform bound also for $\| U_\eps \|_{\Xspace}$.
\end{proof}
 With these uniform estimates we can pass to the limit $\eps \to 0$ to obtain the following theorem.

\begin{theorem} \label{T.conv2}
	Consider a sequence $(U_\eps,v_\eps,w_{\eps})$ of nonnegative solutions to \eqref{eq50}-\eqref{eq:3init} with total mass $ m>0$ and
	under Assumption \eqref{eq:3init}.
	Then there exists a subsequence $\eps \to 0$,  a function $u\in \Xspace$ with $u \in W^{2,1}_p(\Gamma\times (\delta,T))$
	for any $\delta>0$, $1 \le p < \infty$, functions  $w \in L^2(0,T;H^1(\Omega))$ with $w(\cdot,t)\in C^\infty(\Omega)$ for almost all $t\in (0,T)$
	and $\xi \in L^{\infty}(\Gamma_T)$ with $0 \leq \xi \leq 1$, such that
	\begin{equation*}
		U_\eps \weakto u \quad \text{in } \Xspace\,,  \quad
	w_\eps \weakto w \quad \text{in } L^2(0,T;H^1(\Omega)) \quad \mbox{ and } \quad \frac{U_{\eps}}{U_{\eps}+\eps} \weakstarto \xi \quad \mbox{ in } L^{\infty}(\Gamma_T)\,.
	\end{equation*}
  These functions satisfy equations \eqref{ob2,1}, \eqref{ob2,2a} and \eqref{ob2,3} pointwise almost everywhere and the Robin condition in \eqref{ob2,2} in a weak sense.
  Furthermore we have that $u(\cdot,0)=u_0$ on $\Gamma$ in $L^2(\Gamma)$ and that $\int_{\Gamma} u(\cdot,t)\,dS=m$ holds for all $t \in [0,T]$.

	Moreover $u$ and $w$ are  nonnegative with
	$w \in L^\infty(0,T;C^0(\bar \Omega))$
	and for all $\delta>0$ and any $1\leq  p < \infty$ it holds
	\begin{equation*}
		\| u \|_{W^{2,1}_p(\Gamma\times (\delta,T))}+\| w \|_{L^\infty(0,T;C^0(\bar \Omega))}
		\leq C(\delta,T,p)\,.
	\end{equation*}
\end{theorem}

\begin{proof}
	By the uniform bounds provided by Theorem \ref{T.bounds2} we obtain a subsequence and functions $w,u,v, \xi$ such that
	\begin{align}
		w_\eps &\weakto w \quad \text{in } L^2(0,T;H^1(\Omega))  \label{eq65a}\\
		U_\eps &\weakto u \quad \text{in } \Xspace \label{eq65}\\
		v_\eps &\weakstarto v \quad \text{in } L^\infty(0,T;L^2(\Gamma))  \label{eq65c} \\ \frac{U_\eps}{\eps+U_\eps} &\weakstarto \xi \quad \text{in } L^{\infty}(\Gamma_T). \nonumber
	\end{align}
	In particular, we have by the Aubin-Lions Lemma
	that $U_\eps \to u$ in $L^2(\Gamma_T)$.
	The continuity of the trace map $H^1(\Omega)\embeds L^2(\Gamma)$ yields that $w_\eps \weakto w$ in $L^2(\Gamma_T)$.

	We can now  multiply  \eqref{eq50},\eqref{eq51},\eqref{eq52} and \eqref{eq53} by suitable test functions, integrate and pass to the limit $\eps \to 0$, to deduce that
	\begin{align}
		\partial_t u &=\Delta u+cv-a_4 \xi & \text{ on } \Gamma_T \,,\label{eq60}\\
		0 &=-cv+a_4 \xi-a_5v+a_6w &\text{ on } \Gamma_T\,, \label{eq61}\\
		0 &=D\Delta w & \text{ on } \Omega_T \,, \label{eq62} \\
			-D\frac{\partial w}{\partial n}&=-a_5v+a_6w &\text{ on } \Gamma_T\,, \label{eq63}
	\end{align}
are satisfied in a weak sense.
Since the arguments are similar to those used in the proof of Theorem \ref{T.conv1}, we  only consider  $w$ here.
  Multiplying \eqref{eq53} with a test function $\phi\in C^1_c(\bar\Omega\times (0,T))$ and using \eqref{eq52} we obtain
  \begin{align*}
    \int_{\Omega_T} \big(-\eps \partial_t\phi w_\eps +  \nabla\phi\cdot\nabla w_\eps\big)\,dx\,dt &= \int_{\Gamma_T} \phi (a_5 v_\eps - a_6 w_\eps)\,dS\,dt\,.
  \end{align*}
  Passing to the limit $\eps\to 0$ and using the convergence properties obtained above we deduce that
  \begin{align*}
    \int_{\Omega_T} \nabla\phi\cdot\nabla w\,dx\,dt &= \int_{\Gamma_T} \phi (a_5 v - a_6 w)\,dS\,dt,
  \end{align*}
  which implies that \eqref{eq62}, \eqref{eq63}  holds in a weak sense. In particular $w(\cdot,t)$ is harmonic in $\Omega$ for almost all $t\in (0,T)$ and hence smooth inside $\Omega$.

	Finally, it follows exactly in the same way as in \eqref{31b} that
 $\xi u=u$.

  By the uniform bounds \eqref{eq55} on $w_\eps$ and $v_\eps$ we obtain
	$	\int_{\Omega} \eps w_\eps(\cdot,t)\,dx + \int_{\Gamma} \eps v_\eps(\cdot,t) \,dS \to 0$, which together with \eqref{eq54},\eqref{eq65} yields
	$\int_{\Gamma} u(\cdot,t)\,dS=m$ for almost all $t$. Since $u\in \Xspace\embeds C^0([0,T];L^2(\Gamma))$ this equality even holds for all $t\in [0,T]$.
	Since $0 \le \frac{U_\eps}{\eps+U_\eps} \le 1$ the corresponding bounds for $\xi$ follow.
  Furthermore, by \eqref{eq55} and \eqref{eq65a}, \eqref{eq65}, \eqref{eq65c} we deduce
	\begin{equation*}
		\| u \|_{\Xspace}+\| v \|_{L^\infty(0,T;L^2(\Gamma))}+\| w \|_{L^2(0,T;H^1(\Omega))}
		\leq C\,.
	\end{equation*}
	To improve these bounds, we test for $p>2$, equation  \eqref{eq62} with $(k_pw)^{p-1}$ , $k_p := \frac{a_6}{a_5}$
	as well as \eqref{eq61} with $v^{p-1}$ and we find almost everywhere in $(0,T)$
	\begin{align*}
		0=&- \int_{\Omega} D{k_p}^{p-1}(p-1)w^{p-2}{\vert \nabla w \vert}^2\,dx +\int_0^T \int_{\Gamma} {k_p}^{p-1}(a_5 v-a_6 w)w^{p-1}\,dS\\
		&- \int_{\Gamma}\big((a_5 v-a_6 w)v^{p-1}-a_4 \xi v^{p-1} +cv^p\big)\,dS\\
		=&- \int_{\Omega} D{k_p}^{p-1}(p-1)w^{p-2}{\vert \nabla w \vert}^2\,dx
		- \int_{\Gamma} a_5 (v-k_pw)\big(v^{p-1}-(k_pw)^{p-1}\big)\,dS\\
		&+ \int_{\Gamma}\big( a_4 \xi v^{p-1}-cv^p\big)\,dS\\
	 \le & \int_{\Gamma} \big(a_4 \xi v^{p-1}-cv^p\big)\,dS\,.
	\end{align*}
	Thus, using Young's inequality, $c\geq c_0$ and $|\xi|\leq 1$ we conclude
	\begin{equation}
		 \int_{\Gamma} v^p \,dS \le C \quad\text{ almost everywhere on  }(0,T)\label{eq70}
	\end{equation}
	and hence $v$ is bounded in $L^\infty(0,T;L^p(\Gamma))$ for any $1 \le p < \infty$.
	By \cite{Nitt11} and \eqref{eq62} we obtain  for some $\gamma >0$ and for almost all $t\in (0,T)$ that $w(t) \in C^{0,\gamma}(\Omega)$, with
	\begin{equation*}
		\| w(t) \|_{C^{0,\gamma}(\Omega)} \le C\Big(\| v(t) \|_{L^p(\Gamma)} + \|w(\cdot,t)\|_{L^2(\Omega)}\Big)
	\end{equation*}
	for any $p>2$. Therefore, this estimate combined with \eqref{eq70} yields that $w \in L^\infty(0,T;C^0(\bar \Omega))$ for any $p\in [1,\infty)$.
	Finally, by parabolic $L^p-$regularity for \eqref{eq60}, see the arguments in the proof of Theorem \ref{T.conv1},
	we deduce that
	$	\| u \|_{W^{2,1}_p(\Gamma\times (\delta,T))} \le C$ for any $\delta>0$, $1 \le p < \infty$.

	Finally, we  observe that \eqref{eq61} is equivalent to $v=\frac{1-g}{a_5}\big ( a_4 \xi + a_6w\big )$.
	Using this,
	it is easy to see that \eqref{eq60} - \eqref{eq63} are equivalent to \eqref{ob2,1}-\eqref{ob2,2}.
\end{proof}

The system \eqref{ob2,1}-\eqref{ob2,2} can be formulated as an obstacle-type problem in terms of $u$ and $\xi$ only. This formulation will be most convenient for the analysis in Section \ref{4} and contains a non-local operator that we introduce now.
Consider for $s\in L^2(\Gamma)$ and $h\in L^\infty(\Gamma)$, $h\geq 0$, $|\{h>0\}|>0$, the solution $z$ of
\begin{equation}
	0 =\Delta z \text{ in } \Omega, \qquad
	\frac{\partial z}{\partial n}  + hz = s \text{ on }\Gamma. \label{eq4,3a}%
\end{equation}
This defines a linear operator $L_h:L^2(\Gamma)\to H^1(\Omega)$ via $L_hs:=z$. We collect some properties of the operator $L_h$.

\begin{lemma}
\label{L3.1}
Let $h\in L^\infty(\Gamma)$, $h\geq 0$, $|\{h>0\}|>0$, be given. Then the following hold.
\begin{enumerate}
  \item $L_h:L^2(\Gamma)\to H^1(\Omega)$ is continuous.
  \item $L_h:L^2(\Gamma)\to L^2(\Gamma)$ is self-adjoint , that is
  \begin{equation}\label{Lselfadjoint}
    \int_\Gamma s_1 L_h(s_2)\,dS = \int_\Gamma L_h(s_1)s_2\,dS\,.
  \end{equation}
  \item  It holds
  \begin{equation}\label{Lidentity}
    L_hh=1\,.
  \end{equation}
  \item $h\mapsto L_h$ is monotone decreasing in the following sense: For any $h_1,h_2\in L^\infty(\Gamma)$ with $0\leq h_1\leq h_2$ we have
  \begin{equation}\label{Lmonotoneh}
    L_{h_1}(s)\geq L_{h_2}(s) \quad\text{ for all }s\in L^2(\Gamma),\, s\geq 0\,.
  \end{equation}
  \item $L_h$ is positive, more precisely there exists a positive constant $c=c(h,\Omega)$ such that for all $s\geq 0$
  \begin{equation}
    L_h(s) \geq c\int_\Gamma s\,dS\quad\text{ in }\overline{\Omega}.
    \label{Lpositive}
  \end{equation}
\end{enumerate}
\end{lemma}
\begin{proof}
We first have
\begin{equation*}
  \int_\Omega |\nabla z|^2\,dx = \int_\Gamma \big(-h|z|^2 + sz\big)\,dS \leq - \int_\Gamma h|z|^2\,dS  + \|s\|_{L^2(\Gamma)} \|z\|_{L^2(\Gamma)}.
\end{equation*}
Since there holds a generalized Poincar{\'e} inequality in $\{\zeta\in H^1(\Omega): \int_\Gamma h\zeta^2 \leq 1\}$ we deduce
\begin{equation*}
  \|z\|_{H^1(\Omega)}^2 \leq C \Big(\int_\Omega |\nabla z|^2 + \int_\Gamma h|z|^2 \Big) \leq C\|s\|_{L^2(\Gamma)} \|z\|_{H^1(\Omega)},
\end{equation*}
from which $\|z\|_{H^1(\Omega)}\leq C\|s\|_{L^2(\Gamma)}$ and the desired continuity of $L_h$ follow.

The second statement is obtained from
\begin{equation*}
  \int_\Gamma \big(s_1 L_h(s_2) - L_h(s_1)s_2\big)\,dS = \int_\Omega \big(z_2\Delta z_1 - z_1\Delta z_2\big)\,dx = 0\,.
\end{equation*}
The third property is easily verified from the definition of $L_h$.

We next prove that $L_h$ is non-negative, i.e.
\begin{equation}\label{Lnonneg}
  s\geq 0\,\implies\, L_hs \geq 0.
\end{equation}
In fact, with $z:=L_hs$, by a partial integration we deduce
\begin{equation*}
  0 = -\int_\Omega z_-\Delta z\,dx = \int_\Omega |\nabla z_-|^2\,dx - \int_\Gamma \big(hz_-^2+sz_-\big) \,dS \geq 0.
\end{equation*}
Hence $z_-=0$ almost everywhere in $\Omega$ and $z\geq 0$.

We now verify \eqref{Lmonotoneh}. Let $z_1=L_{h_1}(s)$, $z_2=L_{h_2}(s)$. Then
\begin{equation*}
	0 =\Delta (z_1-z_2) \quad  \text{ in } \Omega, \qquad
	\frac{\partial (z_1-z_2)}{\partial n}  + h_1(z_1-z_2) =  z_2(h_2-h_1)\geq 0  \quad \text{ on }\Gamma. 
\end{equation*}
Then \eqref{Lnonneg} ensures that $z_1\geq z_2$.

We finally prove \eqref{Lpositive}. Therefore fix $h\geq 0$, $s\geq 0$, let $m:= \|h\|_{L^\infty(\Gamma)}$ and $\zeta :=L_m s$, i.e.
\begin{equation}
  \Delta \zeta =0 \quad\text{ in }\Omega,\qquad
  \frac{\partial \zeta}{\partial n}  + m\zeta = s\quad\text{ on }\Gamma.
  \label{eq:zeta}
\end{equation}
Then $z:=L_hs\geq L_ms=\zeta$ by \eqref{Lmonotoneh} and to prove \eqref{Lpositive} it suffices to show that there exists $\kappa>0$ with
\begin{equation}
  \zeta \geq \kappa \int_\Gamma s\,dS \,. \label{zeta}
\end{equation}

In the first step of the proof of this inequality we show that for any $K\ssubset \Omega$ there exists a constant $c_1=c_1(K)$ such that
\begin{equation}
  \zeta \geq \frac{c_1}{m}\int_\Gamma s\,dS \qquad\text{ in }K. \label{c1K}
\end{equation}
To prove this estimate consider for $x\in K$ the Green's function $G(x,y)$, i.e. the solution of
\begin{equation*}
  -\Delta G(x,\cdot)=\delta_x \quad\text{ in } \calD'(\Omega),\qquad
  G(x,\cdot)=0 \quad\text{ on }\Gamma.
\end{equation*}
By the positivity of $G$ we derive from the Hopf maximum principle that $\frac{\partial}{\partial n }G(x,y)<0$ for all $x\in K$, $y\in\Gamma$. Since $K\times\Gamma$ is compact and $\frac{\partial}{\partial n }G$ is continuous due to the smoothness of $\Gamma$ we even obtain the existence of $c_1=c_1(K,\Omega)>0$ such that
\begin{equation}
  \frac{\partial}{\partial n }G(x,y) \leq -c_1 \quad\text{ for all }
  x\in K, y\in \Gamma.
  \label{eqGF}
\end{equation}
The representation formula in terms of the Green's function implies that for all $x\in K$
\begin{align*}
  \zeta(x) &= -\int_{\Gamma} \frac{\partial}{\partial n }G(x,\cdot)\zeta\,dS \geq c_1 \int_{\Gamma }\zeta\,dS \,
  =\, \frac{c_1}{m}\int_\Gamma s\,dS\,,
\end{align*}
where the last equality follows from \eqref{eq:zeta}. This proves \eqref{c1K}.

We now use \eqref{c1K} to prove a bound from below for $\zeta$ in the whole set $\Omega$. By the smoothness of $\Gamma$ there is a uniform radius $\varrho>0$ such that for any $y\in\Gamma$ an interior sphere condition is satisfied for a ball $B(z_y,2\varrho) \subset\Omega$.
Moreover $\varrho$ can be chosen such that $\bigcup_{y\in\Gamma} B(z_y,\varrho)\ssubset \Omega\setminus K$ for some compact set $K\subset\Omega$ such that $\partial K$ is smooth and $K$ has nonempty interior.
Denote by $K_1$ the closure of $\bigcup_{y\in\Gamma} B(z_y,\varrho)$. Then in particular $K_1\ssubset \Omega\setminus K$.

We then consider the solution $\tilde\zeta$ of
\begin{equation*}
  \Delta\tilde\zeta = 0 \,\text{ in } \Omega\setminus K,\qquad
  \tilde\zeta = \zeta\,\text{ on }\partial K,\qquad
  \frac{\partial\tilde\zeta}{\partial n } +m \tilde\zeta =0  \,\text{ on }\Gamma.
\end{equation*}
As in the proof of \eqref{Lnonneg} we deduce that $\tilde\zeta\leq\zeta$ and by the maximum principle that $\tilde\zeta\geq 0$.

We claim that
\begin{equation}
 \tilde\zeta \geq \tilde \kappa \int_\Gamma s\,dS
 \quad\text{ in }\Omega\setminus K\,
 \label{eq:tkappa}
\end{equation}
holds. By \eqref{c1K} and $\zeta\geq \tilde\zeta$ this eventually  justifies \eqref{zeta}.

We consider the Green's function $\tilde G$ of $\Omega\setminus K$. Similar as above we obtain that there exists $\tilde c_2>0$ such that
\begin{equation}
  \frac{\partial}{\partial\nu}\tilde G(x,y) \leq -\tilde c_2 \quad\text{ for all }
  x\in K_1, y\in\partial \big(\Omega\setminus K\big),
  \label{eqGFtilde}
\end{equation}
where $\nu$ denotes the outer unit normal field of $\Omega\setminus K$. By the representation formula and the non-negativity of $\tilde \zeta$ we further deduce that for all $x\in K_1$
\begin{align}
  \tilde\zeta(x) = -\int_{\Gamma} \frac{\partial}{\partial \nu }\tilde G(x,\cdot)\tilde \zeta\,dS
  -\int_{\partial K} \frac{\partial}{\partial\nu}\tilde G(x,\cdot)\tilde \zeta\,dS
  &\geq  \tilde c_2\int_{\partial K}  \tilde\zeta\,dS
  \notag\\
  &=   \tilde c_2\int_{\partial K}  \zeta\,dS\,
  \geq\, c_2\frac{c_1}{m}\int_\Gamma s\,dS\,, \label{eq:bdtzeta}
\end{align}
where $c_2=\Ha^{n-1}(\partial K)\tilde c_2$ and where we have used \eqref{c1K} in the last step.

Moreover, the harmonic function $\tilde\zeta$ attains its minimum on $\partial K\cup\Gamma$.
If the minimum is attained on $\partial K$ we have $\tilde\zeta\geq c_1 \int_\Gamma s$ by \eqref{c1K} and conclude that \eqref{eq:tkappa} holds.
If on the other hand the minimum is attained in a point $y_0\in\partial \Omega$ the Hopf boundary point lemma (cf.~the proof of Lemma 3.4 in \cite{GiTr01})
imply that
\begin{equation*}
  \frac{\partial\tilde\zeta}{\partial n }(y_0) \leq -c_3\big(\min_{K_1} \tilde\zeta - \tilde\zeta(y_0)\big)
\end{equation*}
for some positive constant $c_3=c_3(\varrho)$. Using the Robin boundary condition for $\tilde \zeta$ we deduce that
\begin{equation*}
  m \tilde\zeta(y_0) \geq c_3 \big(\min_{K_1} \tilde\zeta - \tilde\zeta(y_0)\big),
\end{equation*}
hence
\begin{equation*}
 \inf_{\Omega\setminus K}\tilde\zeta \geq \frac{c_3}{m+c_3}\min_{K_1} \tilde\zeta \geq  \frac{c_1c_2c_3}{m(m+c_3)}\int_\Gamma s\,dS\,,
\end{equation*}
where we have used \eqref{eq:bdtzeta} in the last step.

This shows \eqref{eq:tkappa} and finishes the proof of \eqref{Lpositive}.
\end{proof}

\begin{proposition}\label{P.equivalence}
	Let $(u,w,\xi)$ be nonnegative functions with $\int_\Gamma u=m> 0$, the same regularity as in Theorem \ref{T.conv2} and with $0\leq\xi\leq 1$ almost everywhere in $\Gamma_T$. Then the following statements are equivalent:
	\begin{enumerate}
		\item $(u,w,\xi)$ satisfies \eqref{ob2,1}-\eqref{ob2,2}.
		\item $(u,\xi)$ satisfies
		\begin{align}
			\partial_t u &= \Delta u-a_4(1-g)\xi+\ell g L_{\ell g}\big(a_4(1-g)\xi\big), \quad u\xi=u \text{ a.e on } \Gamma_T, \label{eq75}
		\end{align}
    where $\ell=\frac{a_6}{D}$, and $w$ is determined by
    \begin{equation}
     w = \frac{\ell}{a_6}L_{\Chi \ell g} \big(a_4\Chi (1-g)\big)\qquad \mbox{ a.e.~on } \Gamma_T\,, \label{eq75a}
    \end{equation}
    with $\Chi=\Chi_{\{u>0\}}$.
	\end{enumerate}
\end{proposition}

\begin{proof}
Due to Stampacchia's Lemma and the regularity of $u$ we have that $a_4 (1-g)\xi= a_6 g w$ holds almost everywhere in $\{u=0\}$. Hence
$D \frac{\partial w}{\partial n} + \Chi a_6 g = a_4 \Chi (1-g)$ and thus \eqref{eq75a} follows.
\end{proof}

\begin{remark}[Infinite cytosolic diffusion limit]
  In \eqref{eq75}, \eqref{eq75a} the parameter $D$ has been substituted by $\ell$. The limit $D\to\infty$ is equivalent to $\ell\to 0$. From the definition of the operator $L_h$ we observe that $z_\ell:= \ell L_{\ell g}(s)$ solves
  \begin{equation*}
  	0 =\Delta z_\ell \text{ in } \Omega, \qquad
  	\frac{\partial z_\ell}{\partial n}  + g\ell z_\ell = \ell s \text{ on }\Gamma.
  \end{equation*}
  We then obtain an estimate
  \begin{equation*}
    \int_\Omega |\nabla z_\ell|^2\,dx
    = \ell \int_\Gamma \big(sz_\ell - gz_\ell^2\big)\,dS
    \leq \ell \|s\|_{L^2(\Gamma)}\|z_\ell\|_{L^2(\Gamma)} - c\ell \|z_\ell\|_{L^2(\Gamma)}^2 \leq \frac{\ell}{2c}\|s\|_{L^2(\Gamma)}^2
  \end{equation*}
  and deduce that $\ell L_{\ell g}(s)$ becomes constant over $\Gamma$ with $\ell\to 0$. This observation shows that \eqref{eq75} reduces to \eqref{ob1,1} in the infinite cytosolic diffusion limit.
\end{remark}

\begin{remark}[Characterization of $\xi$, $w$] \label{rem:xiw}
  In the formulation of \eqref{eq75},\eqref{eq75a} we remark that $\xi$ and $w$ are already determined by $u$. In fact, we have
  \begin{equation}
   \xi(\cdot, t) =
   \begin{cases}
     1 &\quad\text{ a.e.~in } \{u(\cdot,t)>0\},\\
     \frac{a_6 w g}{a_4(1-g)}&\quad \text{a.e.~in } \{u(\cdot,t)=0 \}
   \end{cases}\,
   \label{eq75b}
 \end{equation}
 and $\xi$ is determined by $u,w$. By \eqref{eq75a} we see that $w$ is determined by $u$.

 Notice that the characterization \eqref{eq75a} is analogous to the second formula in \eqref{alpha}. We further remark that we have different representations for the function $w$ in the same manner that we have different characterizations of $\alpha$ (see \eqref{alphaA}). In particular we have also the following characterization in terms of an arbitrary measurable set $A\subset\Gamma$ containing $\{u>0\}$,
 \begin{equation}
   w = \frac{\ell}{a_6}L_{\Chi_A \ell g} \big(a_4(1-g)\Chi_A \xi\big)\qquad \mbox{ on } \Gamma\,. \label{eq75aa}
 \end{equation}
\end{remark}

From now on we set without loss of generality $a_4=a_6=1$.

\section{The reduced model for infinite cytosolic diffusion $D=\infty$ \label{3}}

\subsection{Uniqueness of solutions \label{3.1}}

\begin{theorem}\label{T.uniqueness1}
  Let $(u_1,\xi_1,\alpha_1)$ and $(u_2, \xi_2,\alpha_2)$ be two different solutions of \eqref{ob1,1}-\eqref{ob1,2} with $u_k\in\Xspace$, $\xi_k\in L^\infty(\Gamma_T)$, $\alpha_k\in L^\infty(0,T)$, $k=1,2$. Then
  \begin{equation*}
    t\mapsto \int_{\Gamma}  (u_1-u_2)_+(\cdot,t)\,dS \text{ is decreasing on } [0,T].
  \end{equation*}
	In particular, given $u_0 \in L^2(\Gamma)$ with  $u_0 \geq 0$, there exists at most one solution $(u,\xi,\alpha)$ of \eqref{ob1,1}-\eqref{ob1,3} with $u\in\Xspace$, $\xi\in L^\infty(\Gamma_T)$, $\alpha\in L^\infty(0,T)$.
\end{theorem}

\begin{proof}
	Any solution satisfies in addition $u \in W^{2,1}_p(\Gamma\times (\delta,T))$
	for any $\delta>0$, $1 \le p < \infty$.

  By the regularity of $u_1,u_2$ the function $(u_1-u_2)_+$ belongs to $W^{1,p}(\Gamma_T)$ for any $1\leq p\leq \infty$ and
  \begin{equation}
    \partial_t (u_1-u_2)_+= \Chi_{\{u_1>u_2\}}\partial_t (u_1-u_2). \label{eq:1001}
  \end{equation}
  In particular the weak derivative $\frac{d}{dt} \int_{\Gamma} (u_1-u_2)_+\,dS$ exists as an $L^p(0,T)$ function and hence almost everywhere in $(0,T)$.

  Furthermore, for almost all $t\in (0,T)$ we have $(u_1-u_2)(\cdot,t)\in W^{2,p}(\Gamma)$ and Kato's inequality \cite{Kato72}
  implies that $\Chi_{\{u_1>u_2\}}\Delta (u_1-u_2)\leq \Delta (u_1-u_2)_+$ in the sense of distributions.
  We therefore obtain, with $\one_\Gamma$ denoting the constant function with value $1$ on $\Gamma$,
  \begin{equation}
    \int_{\{u_1>u_2\}} \Delta (u_1-u_2)
    \leq  \langle \Delta (u_1-u_2)_+,\one_\Gamma\rangle  = 0.
    \label{eq:1002}
  \end{equation}
  This justifies the following computations for almost all $t\in (0,T)$. We drop in the following in most places the argument $t$.

  Integrating the equation for the difference $u_1-u_2$ over $\{u_1>u_2\}$ and using \eqref{eq:1001}, \eqref{eq:1002} yields
	\begin{align}
		\frac{d}{dt} \int_{\Gamma} (u_1-u_2)_+\,dS
    &=\int_{\{u_1>u_2\}}\partial_t(u_1-u_2)\,dS
    \notag\\
		&=\int_{\{u_1>u_2\}} \Delta(u_1-u_2)\,dS-\int_{\{u_1>u_2\}} (1-g)(\xi_1-\xi_2)\,dS
    \notag\\
		&\quad +\int_{\{u_1>u_2\}} g(\alpha_1 -\alpha_2)\,dS \notag\\
    &\leq -\int_\Gamma \Chi_+
		(1-g)(\xi_1-\xi_2)\,dS +(\alpha_1 -\alpha_2)\int_\Gamma \Chi_+  g\,dS\,,
    \label{star}
	\end{align}
  where we let $\Chi_+:=\Chi_{\{u_1>u_2\}}$.

	We next rewrite the difference $\alpha_1-\alpha_2$.
  Almost everywhere in $\{u_1=0=u_2\}$ by Stampacchia's Lemma it holds
	\begin{equation*}
		\Delta u_1=\Delta u_2=0\quad  \mbox{ and } \quad \partial_t u_1=\partial_t u_2=0
	\end{equation*}
	which yields due to \eqref{ob1,1},
	\begin{equation}
		(\alpha_1-\alpha_2)g=(1-g)(\xi_1-\xi_2) \label{neweq8}
	\end{equation}
	almost everywhere in $\{u_1=u_2=0\}$.
	We use the notation $\Chi:=\Chi_{\{u_1+u_2>0\}}$ and derive thanks to \eqref{alpha} that
	\begin{align}
		(\alpha_1-\alpha_2) \int_{\Gamma}g\,dS
		&= \int_\Gamma \Chi  (1-g)(\xi_1-\xi_2)\,dS+ (\alpha_1-\alpha_2)\int_\Gamma (1-\Chi) g\,dS \nonumber
	\end{align}
	and thus
	\begin{align}
		(\alpha_1-\alpha_2) \int_\Gamma \Chi   g\,dS
		&=\int_\Gamma \Chi  (1-g)(\xi_1-\xi_2)\,dS \label{neweq9}
	\end{align}

	Plugging \eqref{neweq9} into \eqref{star} we find
	\begin{align}
		\frac{d}{dt} \int_{\Gamma} (u_1-u_2)_+\,dS  \le &\frac{1}{ \int_\Gamma \Chi   g\,dS}
		\Big(-\int_\Gamma \Chi   g\,dS  \int_\Gamma \Chi_+  (1-g)(\xi_1-\xi_2)\,dS \nonumber \\
    &\qquad +\int_\Gamma \Chi  (1-g)(\xi_1-\xi_2)\,dS
		\int_\Gamma \Chi_+  g\,dS \Big)\,. \label{mon.for}
	\end{align}
  For the term on the right-hand side in brackets we further obtain
	\begin{align}
		\Big (\dots \Big)=&-\int_\Gamma (\Chi-\Chi_+)g\,dS \int_\Gamma \Chi_+(1-g)(1-\xi_2)\,dS \nonumber \\
		& -\int_\Gamma \Chi_+ g\,dS \int_\Gamma (\Chi-\Chi_+)(1-g)(1-\xi_1)\,dS\leq 0\,,
    \label{mon.for.b}
	\end{align}
	where we have used that $\Chi-\Chi_+\geq 0$, that $\xi_1=1$ in $\{\Chi_+>0\}$ and that $\xi_2=1$ in $\{\Chi-\Chi_+>0\}$.

  This shows that $t\mapsto \int_{\Gamma}  (u_1-u_2)_+(\cdot,t)\,dS$ is decreasing in time.
	Moreover, since $u_1,u_2\in C^0([0,T];L^2(\Gamma))$ we deduce that $t\mapsto  \int_{\Gamma}  (u_1-u_2)_+(\cdot,t)\,dS$ is continuous on $[0,T]$, and in particular vanishes at $t=0$. This proves that $(u_1-u_2)_+=0$ on $\Gamma_T$, hence $u_1\leq u_2$.
	By the symmetry of the argument, we also have $u_2\leq u_1$, which gives the desired contraction property and the uniqueness for $u$. The uniqueness of $\xi$ and $\alpha$ then easily follows from \eqref{alpha}
	and \eqref{xi}.
\end{proof}

\subsection{Global stability of steady states \label{3.2}}
The results of the previous sections show that for any given initial data with mass $m>0$ there exists a unique solution of \eqref{ob1,1}-\eqref{ob1,3} for all times $t\geq 0$.
We now consider the case that $c=c(x)$ does not depend on time, hence $g=g(x)$ is time-independent, too. The existence and uniqueness of stationary states for any prescribed mass was proved in \cite{NiRV20}.
The goal of this section is to prove that $(u,\xi,\alpha)(\cdot,t)$ converge with $t\to\infty$ to the unique steady state $(\stern  u,\stern\xi,\stern\alpha)$ with the same mass $m$.

In the following we consider $\Gamma_1=\Gamma \times (0,1)$ and denote by $\shift_t u:\Gamma_1\to\R$ the function defined by $(\shift_tu)(x,s):=u(x,s+t)$. The functions  $\shift_t\xi, \shift_t\alpha$ are defined analogously. We denote the constant function with value $(\stern  u,\stern\xi,\stern\alpha)$ on $(0,1)$ again by $(\stern  u,\stern\xi,\stern\alpha)$.
\begin{theorem}\label{T.stability1}
 Consider the unique solution $(u,\xi,\alpha)$ of \eqref{ob1,1}-\eqref{ob1,3} and the stationary solution $(\stern  u,\stern\xi,\stern\alpha)$ with the same mass, that is the unique solution of
	\begin{align}
		-\Delta \stern u &=-(1-g)\stern\xi + \stern\alpha g\,, \qquad
		\stern u \geq 0\,,\quad
		0\leq\stern \xi\leq 1\,, \quad \stern \xi \stern u = \stern u \,, \label{eq:steady1}\\
		\int_{\Gamma} \stern u \,dS &= m\,.\label{eq:steady2}
	\end{align}
 Then $(u,\xi,\alpha)$ converges with $t\to\infty$ to $(\stern  u,\stern\xi,\stern\alpha)$, more precisely
 \begin{equation}
   \shift_t u \weakto \stern u \text{ in }W^{2,1}_p(\Gamma_1)\,,\quad
   \shift_t\xi \weakstarto \stern\xi \text{ in }L^\infty(\Gamma_1)\,,\quad
   \shift_t\alpha \weakstarto \stern\alpha \text{ in }L^\infty(0,1)\,.
   \label{eq:convergence1}
 \end{equation}

 Moreover, $\shift_tu$ converges with $t\to\infty$ uniformly on $\Gamma$ to $\stern u$.
\end{theorem}
\begin{proof}
	We consider for $k\in \N$ the functions
	\begin{align*}
		&(u_k,\xi_k,\alpha_k)\in W^{2,1}_p(\Gamma_1)\times L^\infty(\Gamma_1)\times L^\infty(0,1),\qquad
		u_k = \shift_k u,\,\xi_k = \shift_k \xi,\,\alpha_k = \shift_k \alpha\,.
	\end{align*}
	Then these triples are all solutions of \eqref{ob1,1}, \eqref{ob1,2} on $\Gamma_1$ and we deduce from Theorem \ref{T.bounds1} and \eqref{eq2regu} that they are uniformly bounded in
	$W^{2,1}_p(\Gamma_1)\times L^\infty(\Gamma_1)\times L^\infty(0,1)$ for all $p \in [1,\infty)$.
	Hence, there exists $(u_\infty,\xi_\infty,\alpha_\infty)\in W^{2,1}_p(\Gamma_1)\times L^\infty(\Gamma_1)\times L^\infty(0,1)$
	such that for some subsequence $k\to\infty$
	\begin{equation}
		u_k \weakto u_\infty \text{ in } W^{2,1}_p(\Gamma_1)\,,\quad
		\xi_k\weakstarto \xi_\infty \text{ in }L^\infty(\Gamma_1)\,,\quad
		\alpha_k\weakstarto \alpha_\infty \text{ in }L^\infty(0,1).
		\label{eq:convergence}
	\end{equation}
  By the compact embedding $W^{2,1}_p(\Gamma_1)\embeds C^{\alpha,\alpha/2}(\Gamma\times [0,1])\overset{\text{cpct}}{\embeds} C^0(\Gamma\times [0,1])$
  for $p>\frac{5}{2}$, $0<\alpha\leq 2-\frac{5}{p}$ (see \cite[Theorem 1.4.1]{WuYW06}) we deduce that $\lim_{t\to\infty} u(\cdot,t)=u_\infty$ in $C^0(\Gamma)$.

	We therefore can pass in \eqref{ob1,1}, \eqref{ob1,2} (for $u$ replaced by $u_k$) to the limit and deduce that $(u_\infty,\xi_\infty,\alpha_\infty)$ is
	again a solution of \eqref{ob1,1}, \eqref{ob1,2} on $\Gamma_1$.
	We would like to show that this solution is time-independent and coincides with $(\stern u,\stern \xi,\stern \alpha)$.

	Exactly as in \eqref{star}-\eqref{mon.for.b} we can conclude
	\begin{align}
		\frac{d}{dt}\int_\Gamma (u-\stern u)_+\,dS
		&\leq -\int_{\{u>\stern u\}} \big((1-g)(\xi-\stern\xi) - (\alpha-\stern\alpha)g\big)\,dS \leq 0
		\label{eq:dissi01}
	\end{align}
	and thus $t\mapsto \int_\Gamma (u-\stern u)_+(\cdot,t)$ is decreasing.

	By \eqref{eq:convergence} and the monotonicity property \eqref{eq:dissi01}  we deduce that $\lim_{T\to\infty} \int_\Gamma (u-\stern u)_+(T,\cdot)\,dS$ exists and that for any $t\in (0,1)$
	\begin{equation}
		\int_\Gamma (u_\infty(\cdot,t)-\stern u)_+\,dS = \lim_{k\to\infty}\int_\Gamma (u_k(\cdot,t)-\stern u)_+\,dS
		= \lim_{T\to\infty}\int_\Gamma (u(\cdot,T)-\stern u)_+\,dS
    \label{eq:tindep}
	\end{equation}
	is independent of $t$.
	Since $(u_\infty,\xi_\infty,\alpha_\infty)$ and $(\stern u,\stern \xi,\stern \alpha)$ are both solutions of \eqref{ob1,1}, \eqref{ob1,2} on $\Gamma_1$ we deduce again,
	as in \eqref{eq:dissi01} that
	\begin{align}
	0 = \frac{d}{dt}\int_\Gamma (u_\infty-\stern u)_+\,dS \leq - \int_{\{u_{\infty} > \stern u\}} \big( (1-g)(1-\stern \xi) - (a_{\infty}-\stern \alpha )\big)\,dS \leq 0
	\label{eq:dissi3}
	\end{align}
	and hence the right-hand side must be zero for almost any $t\in (0,1)$.

	Now assume that there exists $t\in (0,1)$ such that $\alpha_\infty(t)<\stern\alpha$ and such that \eqref{eq:dissi3} holds. Then we deduce that $\{u_\infty(\cdot,t)>\stern u\}$ has measure zero and
	$u_\infty(\cdot,t)\leq\stern u$ almost everywhere, which implies by the equal mass condition that $u_\infty(\cdot,t)=\stern u$. But this further induces $\alpha_\infty(t)=\stern\alpha$ by the second equality in \eqref{alpha}, a contradiction.
	Hence $\alpha_\infty(t)\geq\stern\alpha$ for almost all $t\in (0,1)$.

	In a completely analogous way we can derive that $\alpha_\infty(t)\leq\stern\alpha$ for almost all $t\in (0,1)$, which finally implies $\alpha_\infty=\stern\alpha$ almost everywhere.

	Using this information in \eqref{eq:dissi3} and the analogous inequality for $\frac{d}{dt}\int_\Gamma (\stern u- u_\infty)_+\,dS$  we deduce that $\xi_\infty(\cdot,t)=\stern\xi=1$ in $\{u_\infty(\cdot,t)\neq \stern u\}$.
	In addition they also are equal in $\{u_\infty(\cdot,t)=\stern u>0\}$ and by \eqref{xi} also in $\{u_\infty(\cdot,t)=\stern u=0\}$. Hence $\xi_\infty=\stern\xi$ almost everywhere.

	It therefore remains to prove that $(\xi_\infty,\alpha_\infty)=(\stern\xi,\stern\alpha)$ implies  $u_\infty=\stern  u$.
	This follows from the following lemma, applied to $u_\infty-\stern u$.
\end{proof}

\begin{lemma}\label{lem:heatequation}
  Given $u\in W_{2}^{2,1}(\Gamma_{T})$  with $\partial_{t}u-\Delta u=0$ almost everywhere and
  \begin{equation}
    \int_{\Gamma}u(\cdot,t)\,dS=0\,, \qquad \frac{d}{dt}\int_{\Gamma} u(\cdot,t)_+\,dS=0 \qquad \mbox{ for a.a. } t \in (0,T)
    \label{A2}
  \end{equation}
  it follows that $u\equiv0$.
\end{lemma}
\begin{proof}
  Due to the regularity of $u$ the second identity implies
  \begin{equation}
    \int_{\Gamma}(u(\cdot,t_1))_{+}dS=\int_{\Gamma}(u(\cdot,t_2))_+\,dS \qquad  \text{for any }0<t_{1}<t_{2}\leq T\,.\label{A3}
  \end{equation}
  Using standard smoothing effects we can assume that $u\in C^{\infty}(\Gamma\times (0,T))$.
  In particular we
  have that $t\mapsto u(\cdot,t) $ is continuous in $L^{q}(\Gamma)$ for any $q\in[1,\infty]$.
  We define $\psi$ as the solution of
  \[
    \psi_{t}+\Delta\psi=0\,, \qquad  \psi(\cdot,t_{2})  =\chi_{\{u(\cdot,t_{2})>0\}}\,, \qquad  t_{2}\in(0,T]\,.
  \]
  We notice that the set $\{  u(\cdot,t_{2})  >0\}$ is well defined since $u$ is smooth.
  Classical regularity theory for the heat equation implies that $\psi\in C^0([t_{1},t_{2}];L^{p}(\Gamma))  $
  with $0<t_{1}<t_{2}$ and $1\leq p<\infty$.
  Since  $\psi\in C^{\infty}([t_{1}, t_{2}-\delta]\times\Gamma)  $ for any arbitrarily small $\delta>0$ we can use
  $\psi$ as a test function in the equation for $u$.
  Then, integrating by parts we obtain
  \[
    \int_{\Gamma}u(  \cdot,t_{1})  \psi(\cdot,t_{1})\,dS=\int_{\Gamma}u(\cdot,t_{2}-\delta)  \psi( \cdot,t_{2}-\delta)  dS\,.
  \]
  Using the continuity of the map $t\mapsto u(\cdot,t)$  and
  $t\mapsto\psi(\cdot,t)  $ in $L^{2}(  \Gamma) $ we obtain that $u(\cdot,t_{2}-\delta)  \psi(  \cdot,t_{2}-\delta)$
  converges to $u(\cdot,t_{2})  \psi(\cdot,t_{2})$ in $L^{1}(\Gamma)  $ as $\delta
  \rightarrow0$. Thus
  \[
    \int_{\Gamma}u(\cdot,t_{2}-\delta)  \psi(\cdot,t_{2}-\delta)\, dS \rightarrow\int_{\Gamma}u(\cdot,t_2)
    \psi( \cdot, t_{2})\,  dS=\int_{\Gamma}( u(\cdot,t_2))_+\,dS\qquad \text{ as }\delta\rightarrow0,
  \]
  whence
  \begin{equation}
    \int_{\Gamma}(u(\cdot,t_{2}))_{+}\,dS=\int_{\Gamma}u(\cdot,t_{1})  \psi(  \cdot,t_{1})\,  dS\,.\label{A1}
  \end{equation}
  If $\vert \{  u(\cdot,t_{2})>0\}\vert >0$ we have, since $\int_{\Gamma}u(\cdot,t)\,dS  =0$ for all
  $t\in(0,T]$,  that   $\vert \{u(  \cdot,t_{2})  >0\} \vert <\vert \Gamma\vert$. Therefore, the strong maximum principle implies that for
  any $t_{1}<t_{2}$ we have
  \[
    0<\psi(\cdot,t_{1})  \leq\theta<1
  \]
  where $\theta$ depends on $t_{1}$.  Then
  \begin{align*}
    \int_{\Gamma}u(\cdot,t_{1})  \psi(\cdot,t_{1})\,dS &\leq\int\limits_{\{ u(  \cdot,t_{1})  >0\}  }u(\cdot,t_{1})\psi(\cdot,t_{1})\,dS
    \leq\theta\int\limits_{\{u(\cdot,t_{1})  >0\}  }u(\cdot,t_{1})\,dS=\theta\int_{\Gamma}u(\cdot,t_{1})_+\,dS\,.
  \end{align*}
  Combining this with (\ref{A1}) we obtain $\int_{\Gamma}(u(\cdot,t_2))_+\,dS \leq \theta \int_{\Gamma} u(\cdot,t_1)\,dS$
  which contradicts (\ref{A3}).

  Therefore $\vert\{  u(\cdot,t_{2})  >0\}\vert =0.$ Then, we have that $u(\cdot,t_{2})  \leq0,$ but
  since $\int_{\Gamma}u(\cdot,t_2)\,dS=0$ this implies that $u(\cdot,t_2)\equiv0$. Since $t_2$ was arbitrary this proves $u\equiv 0$.
\end{proof}

\section{The reduced model for finite cytosolic diffusion $D<\infty$  \label{4}}
\medskip
From now on we choose $D=1$. All arguments and calculations for the case $D\neq 1$ are analogue. We recall that we have also set $a_4=a_6=1$, which in particular gives $\ell=1$ in the characterization of Proposition \ref{P.equivalence}.

\subsection{Uniqueness of solutions \label{4.1}}
\medskip

In this section we consider a solution
$(u,w,\xi)$ in $ \Xspace \times L^2(0,T;H^1(\Omega)) \times L^\infty(\Gamma_T)$ of
\begin{align}
	\partial_t u&=\Delta u-(1-g)\xi+gw, \quad u\xi=u,\quad
  u\geq 0  &\text{  on } \Gamma_T \label{eq4,1}\\
	0&=\Delta w \text{ in } \Omega, \quad  \frac{\partial w}{\partial n}=(1-g)\xi-gw & \text{ on } \Gamma_T\,, \label{eq4,2}\\
	u(\cdot,0)&=u_0\, & \mbox{ on } \Gamma\,.\label{eq4,2b}
\end{align}
We recall that
\begin{equation}
  \xi(\cdot,t)=
  \begin{cases}
    1 &\text{a.e.~in } \{u(\cdot,t)>0 \}\\
    \frac{w g}{1-g}(\cdot,t)&\text{a.e.~in } \{u(\cdot,t)=0 \}
  \end{cases}\,.
  \label{newxi}
\end{equation}

In the following we use the operator $L_h$ as defined before Lemma \ref{L3.1}, i.e. for given $h\in L^\infty(\Gamma)$, $h\geq 0$ the function $z=L_hs$ solves
\begin{equation}
	0 =\Delta z \quad\text{ in } \Omega, \qquad
	\frac{\partial z}{\partial n}  + hz = s \quad\text{ on }\Gamma. \label{eq4,3}
\end{equation}

We next prove an $L^1$-contraction property and the uniqueness of solutions.
\begin{theorem}\label{T.uniqueness2}
 Consider two solutions $(u_k,\xi_k,w_k)$, $k=1,2$ of \eqref{eq4,1}-\eqref{eq4,2}. Then
 \begin{equation*}
   t\mapsto \int_{\Gamma} (u_1-u_2)_+(\cdot,t) \,dS \text{ is decreasing on }(0,T).
 \end{equation*}
In particular, given $u_0 \in L^2(\Gamma)$ with $u_0 \geq 0$ and $T>0$, there exists at most one solution $u \in \Xspace$, $\xi \in L^{\infty}(\Gamma_T)$, $w \in L^2(0,T;H^1(\Omega))$ of \eqref{eq4,1}-\eqref{eq4,2b}.
\end{theorem}

\begin{proof}
As above, by parabolic regularity results, we have $u_k\in W^{2,1}_p(\Gamma\times (\delta,T))$ for any $\delta>0$, $1\leq p<\infty$.

Letting $s_k=(1-g)\xi_k$ we have
\begin{equation}
  w_k = L_g s_k. \label{eq4,4}
\end{equation}
In the following we let $\Chi_+=\Chi_{\{u_1>u_2\}}$ and $\Chi=\Chi_{\{u_1+u_2>0\}}$. As in the proof of \eqref{eq75a} we conclude that the difference $w_1-w_2$ satisfies
\begin{equation}
 w_1-w_2 = L_{\Chi g} \big(\Chi (s_1-s_2)\big) \,. \label{eq4,5}
\end{equation}

Following the arguments in the proof of Theorem \ref{T.uniqueness1} we obtain, using also Lemma \ref{L3.1}, that
\begin{align}
  \frac{d}{dt} \int_{\Gamma}& (u_1- u_2)_+\,dS  \leq \int_{\{u_1>u_2\}}\big( -(s_1-s_2)+g(w_1-w_2)\big)\,dS \nonumber \\
  &= \int_{\Gamma} \big( -\Chi_+(s_1-s_2) + \Chi_+ g L_{\Chi g}\big(\Chi (s_1-s_2)\big)\big)\,dS  \nonumber \\
  &= \int_{\Gamma} \big(-\Chi_+(s_1-s_2)L_{\Chi g}(\Chi g) + \Chi_+ gL_{\Chi g}\big(\Chi (s_1-s_2)\big) \big)\,dS  \label{eq4,6}\\
  &=\int_{\Gamma} -\Chi gL_{\Chi g}\big(\Chi_+(s_1-s_2)\big)+ \Chi_+ gL_{\Chi g}\big(\Chi (s_1-s_2)\big) \big)\,dS \nonumber\\
  &= -\int_{\Gamma} \big( \Chi-\Chi_+) gL_{\Chi g}\big(\Chi_+(s_1-s_2)\big)\,dS
  +\int_{\Gamma} \Chi_+ gL_{\Chi g}\big((\Chi-\Chi_+)(s_1-s_2)\big)\,dS\leq 0\,.\nonumber
\end{align}
In the last line we have used in the first term that  $\Chi-\Chi_+\geq 0$ and $\Chi_+(s_1-s_2)\geq 0$ and for the second term that $\Chi-\Chi_+=\Chi_{\{u_2>u_1\}}+\Chi_{\{u_1=u_2>0\}}$,
that $s_1\leq s_2$ on $\{u_2>u_1\}$ and $s_1=s_2$ on $\{u_1=u_2>0\}$.

Applying the same argument to $u_2-u_1$ we find that $\int_{\Gamma}|u_1-u_2|\,dS$ is decreasing in time,
and in particular $u_1=u_2$ since the initial data are the same.

From Remark \ref{rem:xiw} it follows that $w_1=w_2$ and $\xi_1=\xi_2$.
\end{proof}

With similar arguments as in the proof of Theorem \ref{T.uniqueness2} we can also show  uniqueness of steady states for given mass $m$. This result has been shown in \cite{NiRV20} only in the
case that $\Gamma$ is a sphere. In the following Theorem we prove even more, namely a monotonicity result from which uniqueness of steady states follows.

\begin{theorem}[Monotonicity]\label{thm:monotonicity}
Let $(u_1,w_1,\xi_1)$, $(u_2,w_2,\xi_2) \in H^2(\Gamma) \times H^1(\Omega) \times L^{\infty}(\Gamma)$ be solutions  to
\begin{align}
		-\Delta u&=-(1-g)\xi+gw\,, \qquad u\xi=u,\quad
    u\geq 0  &\text{  on } \Gamma \label{eq4s1}\\
		0&=\Delta w \text{ in } \Omega\,, \qquad \qquad   \frac{\partial w}{\partial n}=(1-g)\xi-gw& \text{ on } \Gamma \,, \label{eq4s2}
				\end{align}
with $\int_{\Gamma} u_1\,dS = m_1$ and $\int_{\Gamma} u_2\,dS=m_2$.
Suppose that $m_1 \geq m_2$, then
\[
  u_1\geq u_2,\qquad w_1\geq w_2,\qquad \xi_1\geq \xi_2
  \qquad\text{  on }\Gamma.
\]
\end{theorem}

\begin{proof}
  Again we let $s_k=(1-g)\xi_k$, $\Chi_+=\Chi_{\{u_1>u_2\}}$ and $\Chi=\Chi_{\{u_1+u_2>0\}}$.

  We first show that $u_1\geq u_2$. We integrate the difference of the equations for $u_1$ and $u_2$ over the set $\{u_1>u_2\}$ and obtain, exactly as in \eqref{eq4,6} that
  \begin{align}
    0 & \leq \int_{\{u_1>u_2\}}\big( -(1-g)(\xi_1-\xi_2) +g(w_1-w_2)\big)\,dS \nonumber \\
     &= -\int_{\Gamma} \big( \Chi-\Chi_+) gL_{\Chi g}\big(\Chi_+(s_1-s_2)\big)\,dS
    +\int_{\Gamma} \Chi_+ gL_{\Chi g}\big((\Chi-\Chi_+)(s_1-s_2)\big)\,dS \leq 0\,.
    \label{eq4,9}
  \end{align}
  We now exploit that both integrands in the last line of \eqref{eq4,9} vanish. If $\Chi_+=0$ almost everywhere or $\Chi-\Chi_+=0$ almost everywhere,
  then $u_1\leq u_2$ or $u_1\geq u_2$, respectively, hence $u_1 \geq u_2$ almost everywhere since we have assumed that $m_1\geq m_2$.

  If $\Chi_+$ and $\Chi-\Chi_+$ are both nontrivial we deduce from the positivity of $L_{\Chi g}$, see \eqref{Lpositive}, that $s_1=s_2$ and thus $\xi_1=\xi_2$ in $\{u_1+u_2>0\}$.
  By the first line in \eqref{eq4,9} this in addition implies $w_1=w_2$ in $\{u_1>u_2\}$.

  Testing the difference equation with $(u_1-u_2)_+$ yields
  \begin{equation*}
    0 = \int_\Gamma \big(|\nabla (u_1-u_2)_+|^2 + \big((s_1-s_2)-g(w_1-w_2)\big)(u_1-u_2)_+\big)\,dS = \int_\Gamma |\nabla (u_1-u_2)_+|^2\,dS\,.
  \end{equation*}
  This implies that $(u_1-u_2)_+$ is constant, from which we obtain by $m_1\geq m_2$  that $u_1\geq u_2$.

  \medskip
  The property $u_1\geq u_2$ implies that $\Chi(\xi_1-\xi_2)\geq 0$. Therefore \eqref{eq75a} and the positivity of $L_h$, see \eqref{Lpositive}, imply that
  \begin{equation*}
    w_1-w_2 = L_{\Chi g}\big((1-g)\Chi(\xi_1-\xi_2)) \geq 0.
  \end{equation*}
  Then, using \eqref{eq75b} we finally deduce $\xi_1\geq\xi_2$.
\end{proof}

\subsection{Global stability of steady states \label{4.2}}
Again we assume in this section that $c=c(x)$ does not depend on time, hence $g$ has the same property. We prove the convergence of the obstacle-type problem for finite diffusion to the stationary state with the same mass.
We again denote the shift operator by $\shift_t$, see the definition before Theorem \ref{T.stability1}.

\begin{theorem}\label{T.stability2}
 The unique solution $(u,w,\xi)$ of \eqref{eq4,1}-\eqref{eq4,2b} converges as $t \to \infty$ to the unique stationary solution $(\stern  u,\stern w,\stern\xi)$ of \eqref{eq4s1}-\eqref{eq4s2} with
 $\int_{\Gamma} \stern u\,dS=m=\int_{\Gamma}u_0\,dS$,
 more precisely
 \begin{equation}
   \shift_t u \weakto \stern u \text{ in }W^{2,1}_p(\Gamma_1)\,,\quad
   \shift_t\xi \weakstarto \stern\xi \text{ in }L^\infty(\Gamma_1)\,,\quad
   \shift_t\alpha \weakstarto \stern\alpha \text{ in }L^\infty(0,1)\,.
   \label{eq:convergence3}
 \end{equation}
 In particular, $\shift_tu$ converges with $t\to\infty$ uniformly on $\Gamma$ to $\stern u$.
\end{theorem}

\begin{proof}
Since $(\stern  u,\stern w,\stern\xi)$ is a solution of \eqref{eq4,1}-\eqref{eq4,2} we obtain from Theorem \ref{T.uniqueness2} that 
$t\mapsto \int_{\Gamma} (u(\cdot,t)- \stern u)_+\,dS$ is decreasing and
\begin{equation}
  \lim_{T\to\infty} \int_{\Gamma} (u(\cdot,T)- \stern u)_+\,dS \quad\text{ exists.}
  \label{eq:Tlim}
\end{equation}
We consider for $k\in \N$ the functions
\begin{align*}
	&(u_k,w_k,\xi_k)\in W^{2,1}_p(\Gamma_1)\times L^2(0,1;H^1(\Omega))\times L^\infty(\Gamma_1),\\
	&\big(u_k(\cdot,t), w_k(\cdot,t), \xi_k(\cdot,t)\big) =
	\big(u(\cdot,t+k), w(\cdot,t+k), \xi(\cdot,t+k)\big).
\end{align*}
Then $u_k, w_k, \xi_k$ are uniformly bounded in $W^{2,1}_p(\Gamma_1)\times L^2(0,1;H^1(\Omega))\times L^\infty(\Gamma_1)$ for all $p \in [1,\infty)$.
Hence, there exists $(u_\infty,w_\infty,\xi_\infty) \in  W^{2,1}_p(\Gamma_1)\times L^2(0,1;H^1(\Omega))\times L^\infty(\Gamma_1)$ such that for some subsequence $k \to \infty$
\begin{equation}
  u_k \weakto u_\infty \text{ in } W^{2,1}_p(\Gamma_1)\,,\quad
  w_k \weakto w_\infty \text{ in } L^2(0,1;H^1(\Omega))\,,\quad
	\xi_k\weakstarto \xi_\infty \text{ in }L^\infty(\Gamma_1)\,.
		\label{eq:convergence2}
\end{equation}
As in the proof of Theorem \ref{T.stability1} we deduce that $\lim_{t\to\infty} u(\cdot,t)=\stern u$ in $C^0(\Gamma)$ and that $(u_\infty,w_\infty,\xi_\infty)$ is again a solution of \eqref{eq4,1},\eqref{eq4,2}.
We prove that this solution is time-independent and coincides with $(\stern u,\stern w,\stern \xi)$.

We first deduce from \eqref{eq:Tlim} as in \eqref{eq:tindep} that $t\mapsto \int_{\Gamma} (u_\infty(\cdot,t)- \stern u)_+\,dS$ is independent of $t\in (0,1)$.

Since $(u_\infty,w_\infty,\xi_\infty)$ and $(\stern u,\stern w,\stern \xi)$ are both solutions to \eqref{eq4,1},\eqref{eq4,2} we obtain from \eqref{eq4,6} that
\begin{align*}
  0 &= \frac{d}{dt} \int_{\Gamma} (u_\infty- \stern u)_+\,dS  \\
  &\leq -\int_{\Gamma} \big( \Chi-\Chi_+) gL_{\Chi g}\big(\Chi_+(s_\infty-\stern s)\big)\,dS
  +\int_{\Gamma} \Chi_+ gL_{\Chi g}\big((\Chi-\Chi_+)(s_\infty-\stern s)\big)\,dS\leq 0\,,
\end{align*}
where  $s_\infty=(1-g)\xi_\infty$, $\stern s=(1-g)\stern \xi$, $\Chi_+=\Chi_{\{u_\infty>\stern u\}}$ and $\Chi=\Chi_{\{u_\infty+\stern u>0\}}$. We therefore deduce as for \eqref{eq4,6} that both integrals on the right-hand side are zero.

In this situation we can follow the arguments after \eqref{eq4,9}. Since $u_\infty$, $\stern u$ have the same mass we obtain that $u_\infty=\stern u$ or $s_\infty=\stern s$ on $\{u_\infty+\stern u>0\}$. In the first case the claim is proved.

In the second case we have $\xi_\infty=\stern \xi$ on $\{u_\infty+\stern u>0\}$ and it remains to examine what holds in the region $\{ u_\infty=\stern u=0\}$.
To this end, it is more convenient to show first that $w_\infty=\stern w$. This follows easily from
\eqref{eq75aa}, with $A=\{ u_\infty+\stern u>0\}$. Indeed, since $\xi_{\infty}=\stern \xi$ almost everywhere in $\{u_\infty+\stern u>0\}$, we deduce that $w_\infty =\stern w$ almost everywhere in $\Gamma \times (0,1)$.
This, combined with \eqref{eq75b} implies that $\xi_{\infty}=\stern \xi$ almost everywhere in $\Gamma \times (0,1)$.
\medskip

What is left to prove is that $(\xi_\infty,w_\infty)=(\stern\xi, \stern w)$ implies $u_\infty=\stern u$. We notice that:
\begin{align*}
  \partial_t (u_\infty-\stern u) &=\Delta (u_\infty-\stern u).
\end{align*}
In addition, $\int_\Gamma u_\infty(\cdot,t) \;dS=\int_\Gamma \stern u(\cdot,t)\;dS$ for all $t\in (0,1)$ and we recall that $\int_\Gamma (u_\infty -\stern u)_+(\cdot,t) \,dS$ is constant for all $t \in (0,1)$.
Therefore, it follows from Lemma \ref{lem:heatequation} that $u_\infty=\stern u$.

\end{proof}

\bigskip
{\bf Acknowledgments.}
The authors acknowledge the support of the Hausdorff Center of Mathematics at the University of Bonn.

\bigskip

\end{document}